\documentclass[11pt,a4paper,american]{article}
\usepackage{babel}
\usepackage[T1]{fontenc}
\usepackage[utf8]{inputenc}
\usepackage{fullpage}
\usepackage[babel]{microtype}
\usepackage{csquotes}
\usepackage{authblk}
\usepackage{hyperref}
\usepackage{enumitem}
\usepackage{amsmath}
\usepackage{mathtools}
\usepackage{amsthm}
\usepackage{thmtools}
\usepackage{dsfont}
\usepackage{amssymb}
\usepackage{xcolor}
\usepackage{zref-clever}
\usepackage{mleftright}

\hypersetup{
    pdfencoding = auto,
    colorlinks = true,
    linkcolor = blue,
    citecolor = teal,
    hypertexnames = false,
    pdfauthor = {Marco Pozza and Alfonso Sorrentino},
    pdftitle = {Stochastic homogenization of the Hamilton–Jacobi equation on manifolds},
    bookmarksnumbered = true,
}

\declaretheorem[style=plain,parent=section,title=Theorem]{theo}
\declaretheorem[style=plain,sibling=theo,title=Proposition]{prop}
\declaretheorem[style=plain,sibling=theo,title=Corollary]{cor}
\declaretheorem[style=plain,sibling=theo,title=Lemma]{lem}
\declaretheorem[style=definition,sibling=theo,title=Definition]{defin}
\declaretheorem[style=remark,sibling=theo,title=Remark]{rem}
\declaretheorem[style=plain,numbered=no,title=Main Theorem]{main}

\newlist{Hassum}{enumerate}{1}
\setlist[Hassum]{label=\textbf{(H\arabic*)},ref=\textnormal{\textbf{(H\arabic*)}},font=\normalfont}
\newlist{myenum}{enumerate}{1}
\setlist[myenum]{label=\textbf{\roman*)},ref=\textnormal{\textbf{(\roman*)}},font=\normalfont}
\newlist{mylist}{itemize}{1}
\setlist[mylist]{label=\textbullet,font=\normalfont}

\zcRefTypeSetup{condenum}{
    nocap,
    Name-sg = Condition,
    name-sg = condition,
    Name-pl = Conditions,
    name-pl = conditions,
}
\zcRefTypeSetup{equation}{
    Name-sg =,
    name-sg =,
    Name-pl =,
    name-pl =,
    Name-sg-ab =,
    name-sg-ab =,
    Name-pl-ab =,
    name-pl-ab =,
    refbounds = {\textup{(},,,\textup{)}},
    reffont = \normalfont,
}
\zcRefTypeSetup{item}{nocap}

\zcLanguageSetup{american}{cap}

\zcsetup{
    nocomp,
    rangesep = {--},
    nameinlink = false,
    countertype = {
        theo = theorem,
        prop = proposition,
        cor = corollary,
        lem = lemma,
        defin = definition,
        rem = remark,
        myenumi = item,
        Hassumi = condenum,
    },
}

\DeclareMathOperator{\diam}{diam}
\DeclareMathOperator{\dis}{dis}
\DeclareMathOperator{\Orm}{O}

\newcommand{\ovB}{\overline{B}}
\newcommand{\Bcal}{\mathcal{B}}
\newcommand{\Ccal}{\mathcal{C}}
\newcommand{\Crm}{\mathrm{C}}
\newcommand{\Drm}{\mathrm{D}}
\newcommand{\Eds}{\mathds{E}}
\newcommand{\Fcal}{\mathcal{F}}
\newcommand{\ovH}{\overline{H}}
\newcommand{\whH}{\widehat{H}}
\newcommand{\ovL}{\overline{L}}
\newcommand{\whM}{\widehat{M}}
\newcommand{\Nds}{\mathds{N}}
\newcommand{\Pds}{\mathds{P}}
\newcommand{\Qds}{\mathds{Q}}
\newcommand{\Rds}{\mathds{R}}
\newcommand{\Tcal}{\mathcal{T}}
\newcommand{\Tds}{\mathds{T}}
\newcommand{\Xcal}{\mathcal{X}}
\newcommand{\Zds}{\mathds{Z}}
\newcommand{\drm}{\mathrm{d}}
\newcommand{\ovd}{\overline{d}}
\newcommand{\ovh}{\overline{h}}
\newcommand{\ovr}{\overline{r}}
\newcommand{\ovOmega}{\overline{\Omega}}
\newcommand{\ovTheta}{\overline{\Theta}}
\newcommand{\whTheta}{\widehat{\Theta}}
\newcommand{\ovvartheta}{\overline{\vartheta}}
\newcommand{\myemail}[2][]{\textsuperscript{#1}\href{mailto:#2}{\texttt{#2}}}

\newcommand{\g}{\mathfrak{g}}
\newcommand{\G}{\mathtt{G}}

\title{Stochastic Homogenization of the Hamilton--Jacobi Equation on Manifolds}
\author[1]{Marco Pozza}
\author[2]{Alfonso Sorrentino}
\affil[1]{Link Campus University, Rome, Italy. \textit{Email~address:}~\myemail{m.pozza@unilink.it}}
\affil[2]{Dipartimento di Matematica, Università degli Studi di Roma ``Tor Vergata'', Rome, Italy. \textit{Email~address:}~\myemail{sorrentino@mat.uniroma2.it}}
\date{}

\begin{document}

    \maketitle

    \begin{abstract}
        This article establishes a stochastic homogenization result for the first order Hamilton--Jacobi equation on a Riemannian manifold $M$, in the context of a stationary ergodic random environment. The setting involves a finitely generated abelian group $ \G$ of rank $b$ acting on $M$ by isometries in a free, totally discontinuous, and co-compact manner, and a family of Hamiltonians $H: T^*M \times \Omega \to \Rds$, parametrized over a probability space $(\Omega, \mathcal{F}, \Pds)$, which are stationary with respect to a $\Pds$-ergodic action of $\G$ on $\Omega$. Under standard assumptions, including strict convexity and coercivity in the momentum variable, we prove that as the scaling parameter $\varepsilon$ goes to $0$, the viscosity solutions to the rescaled equation converge almost surely and locally uniformly to the solution to a deterministic homogenized Hamilton--Jacobi equation posed on $\Rds^b$, which corresponds to the asymptotic cone of $\G$. In particular, this approach sheds light on the relation between the limit problem, the limit space, and the complexity of the acting group. The classical periodic case corresponds  to a randomness set $\Omega$ that reduces to a singleton; other interesting examples of this setting are also described.

        We remark that the effective Hamiltonian $\overline{H}$ is obtained as the convex conjugate of an effective Lagrangian $\overline{L}$, which generalizes Mather's $\beta$-function to the stochastic setting; this represents a first step towards the development of a stationary-ergodic version of Aubry--Mather theory. As a geometric application, we introduce a notion of stable-like norm for stationary ergodic families of Riemannian metrics on $M$, which generalizes the classical Federer--Gromov's stable norm for closed manifolds.
    \end{abstract}

    \section{Introduction}

    The theory of homogenization for first order Hamilton--Jacobi (HJ) equations originated with the seminal work of Lions, Papanicolaou, and Varadhan~\cite{LionsPapanicolaouVaradhan87}, who studied the asymptotic behavior of viscosity solutions to HJ equations associated to $\Zds^n$-periodic Hamiltonians on $\Rds^n\times \Rds^n$, as the spatial variable oscillates faster and faster. Their fundamental result established that the solutions converge to a limit function satisfying a homogenized equation, characterized by an effective Hamiltonian $\overline{H}: \Rds^n\rightarrow \Rds$ that is convex and independent of the spatial variable. This periodic framework has since been extended to more general settings, including problems on closed manifolds~\cite{ContrerasIturriagaSiconolfi14}, corresponding to a periodicity of the space variable with respect to a finitely generated abelian group (see also~\cite{Sorrentino19}, where the non-abelian case has been addressed).\\
    A crucial element in these results is the presence of a group of symmetries acting on the ambient space, hence determining a periodic structure. Under suitable assumptions on this action, periodicity provides a compact fundamental cell: it is this compactness that underpins the convergence process and, importantly, determines the nature of the limit space and the limit problem itself. Remarkably, it turns out that it is the periodicity that dictates the geometry of the limit space, which corresponds to the so-called {\it asymptotic cone} of the acting group.\\
    Extending this theory beyond the periodic case is a significant challenge: it is not clear what should play the role of the fundamental cell, what the appropriate limit space should be, or even if a homogenized limit exists at all.

    In this article, we consider a specific non-periodic setting. Let $M$ be a smooth Riemannian manifold (we denote the associated distance by  $d$)  and let $\G$ be a finitely-generated abelian group acting on $M$ by isometries; we denote by $b$ the rank of $\G$. We are not considering a Hamiltonian on $T^*M$ that is invariant under the action induced by $\G$ on $T^*M$, but we consider a family of Hamiltonians that are ``stochastically''  related by this action. This is modeled by introducing a probability space $(\Omega, \mathcal{F}, \Pds)$ on which $\G$ acts ergodically via measure-preserving transformations $\{\tau_g\}_{g \in \G}$, while the family of Hamiltonians is given in terms of a stochastic Hamiltonian $ H: T^*M \times \Omega \to \Rds $ satisfying the following assumptions: it is a continuous random function, $ \Crm^2 $ and strictly convex in the momentum variable for each $ \omega $, coercive, and, crucially, \emph{stationary} with respect to $ \{\tau_g\}_{g \in \G} $, namely: $H(g\cdot (x, p), \omega) = H(x, p, \tau_g(\omega))$ for all $(x,p)\in T^*M$ and $g \in \G$, where $g\cdot (x, p)$ denotes the action induced by $\G$ on $T^*M$.

    According to~\cite{PapanicolaouVaradhan81}, this framework can be seen as a generalization of the classical periodic, quasi-periodic and almost-periodic cases in $\Rds^n$, where the parameter $\omega$ controls the deviation from periodicity. In Appendix~\ref{sec:exe} we describe how to construct many examples that fit into our setting.

    In this context, we are able to prove a homogenization result. Roughly speaking, the ergodic action of the group $\G$ on the probability space provides a substitute for the missing periodicity in a single realization and allows us to recover some sort of compactness necessary for the asymptotic analysis.

    More precisely, for $ \varepsilon > 0 $, we consider the rescaled manifold $ M_\varepsilon = (M, \varepsilon\, d) $ and the rescaled Hamiltonian $ H_\varepsilon(x, p, \omega) \coloneqq H(x, p/\varepsilon, \omega) $. We investigate the asymptotic behavior, as $ \varepsilon \to 0 $, of solutions to the stochastic first order Hamilton--Jacobi equation
    \begin{equation*}
        \begin{cases}
            \partial_t u^\varepsilon(x,t,\omega) + H_\varepsilon(x, \partial_x u^\varepsilon(x,t,\omega), \omega) = 0, & \text{in } M_\varepsilon \times (0, \infty) \times \Omega, \\
            u^\varepsilon(x,0,\omega) = u_0^\varepsilon(x,\omega), & \text{in } M_\varepsilon \times \Omega,
        \end{cases}
    \end{equation*}
    where the initial data $ u_0^\varepsilon $ are random, uniformly continuous functions.\\

    Our main result can be summarized as follows (we refer to Section~\ref{sec:2.3} for a more precise statement with all underlying hypotheses):\\

    \noindent{\bf Outline of the Main Theorem.} {\it There exists a convex and superlinear effective Hamiltonian $ \overline{H}: \Rds^b \to \Rds $ such that, for almost every $ \omega \in \Omega $, the solutions $ u^\varepsilon(\cdot, \cdot, \omega) $ converge locally uniformly (in some suitable sense) to the unique uniformly continuous solution $ v $ of the homogenized equation
        \begin{equation*}
            \begin{cases}
                \partial_t v(h,t) + \overline{H}(\partial_h v(h,t)) = 0, & \text{in } \Rds^b \times (0, \infty), \\
                v(h,0) = v_0(h), & \text{in } \Rds^b,
            \end{cases}
        \end{equation*}
        provided the initial data $ u_0^\varepsilon(\omega) $ converge appropriately to a limit $ v_0: \Rds^b \to \Rds $. }

    \medskip

    We remark that:
    \begin{itemize}
        \item The limit space $ \Rds^b $, on which the limit problem is posed, emerges as the \emph{asymptotic cone} of the acting group $\G$ (see Section~\ref{conesec}).
        \item The convergence should be properly addressed since it involves functions that are defined on different spaces, which may have very different dimensions (see Section~\ref{sec:3.2}).
        \item The periodic case in~\cite{ContrerasIturriagaSiconolfi14, LionsPapanicolaouVaradhan87} can be seen as a special case, where the parameter space consists of a single element (see Appendix~\ref{sec:exe}).
        \item The effective Hamiltonian $ \overline{H} $ is the convex conjugate of an \emph{effective Lagrangian} $ \overline{L}: \Rds^b \to \Rds $, which generalizes Mather's $ \beta $-function from Aubry--Mather theory in this context. {This generalization of Mather's $\beta$-function to stationary ergodic families of Tonelli Lagrangians is of independent interest and it represents a first step towards the development of an Aubry--Mather theory in this context.} We also remark that in the case of Riemannian metrics, $ \overline{L} $ is related to the \emph{stable norm} on homology. In Appendix~\ref{sec:stablenorm}, we discuss how our investigation allows one to define a {\it stable norm} for families of stationary ergodic Riemannian metrics, overcoming the problem of having a non-compact ambient manifold and the absence of periodicity.
        \item It is crucial to emphasize that in the stationary ergodic setting, the homogenization of Hamilton--Jacobi equations generally fails without the convexity assumption on the momentum variable. In~\cite{Ziliotto17}, Ziliotto constructs an example of a Hamilton--Jacobi equation in the two-dimensional case such that, $\Pds$-almost surely, $u^{\varepsilon}(0,1,\omega)$ does not converge when $\varepsilon$ goes to $0$. Hence, there is no stochastic homogenization for this equation. In this example, the Hamiltonian $H$ satisfies all the standard assumptions of the literature except the convexity with respect to $p$. This counterexample breaks homogenization in a very ingenious way by designing a Hamiltonian where the non-convexity in $p$ interacts with the randomness to create a persistent, non-averaging oscillation.
    \end{itemize}

    \bigskip

    \bigskip

    \subsection{Comparison with previous literature}

    Since the seminal works by Lions, Papanicolaou, and Varadhan~\cite{LionsPapanicolaouVaradhan87} (see also~\cite{Evans89}), an extensive body of literature on the homogenization of the Hamilton–Jacobi equation has emerged in the context of periodic Hamiltonians, with significant extensions to more general settings, like almost periodic media~\cite{Ishii00}, non-Euclidean settings~\cite{ContrerasIturriagaSiconolfi14, Sorrentino19}, problems defined on junctions or periodic networks~\cite{CamilliMarchi23, ForcadelSalazar20, ForcadelSalazarZaydan18, GaliseImbertMonneau15, ImbertMonneau17, PozzaSiconolfiSorrentino24}, or by means of tools coming from symplectic topology (the so-called symplectic homogenization~\cite{Viterbo23}, see also~\cite{MonznerVicheryZapolsky12}).

    The foundations of the theory of stochastic homogenization of partial differential equations in stationary ergodic random environments were laid in the 1980s with the pioneering works of Papanicolaou and Varadhan~\cite{PapanicolaouVaradhan81,PapanicolaouVaradhan82} and Kozlov~\cite{Kozlov85}, who settled the linear case, while Dal Maso and Modica~\cite{DalMasoModica86,DalMasoModica86_2} extended the analysis to general variational problems.

    The stochastic homogenization of convex first-order Hamilton--Jacobi equations in Euclidean settings was first proved independently by Souganidis~\cite{Souganidis99} and Rezakhanlou and Tarver~\cite{RezakhanlouTarver00}. We also recall the contribution by Lions and Souganidis~\cite{LionsSouganidis10}, who introduced a more direct proof of homogenization in probability. Davini and Siconolfi~\cite{DaviniSiconolfi09, DaviniSiconolfi11, DaviniSiconolfi11_2, DaviniSiconolfi16} brought an interesting connection to Mather theory and weak KAM theory, providing new geometric insights into the problem. Armstrong and Souganidis~\cite{ArmstrongSouganidis12, ArmstrongSouganidis13} provided almost sure homogenization results based on a metric approach. The connection between the metric approach to homogenization and the theory of first-passage percolation played a crucial role in the works of Armstrong, Cardaliaguet, and Souganidis~\cite{ArmstrongCardaliaguet15, ArmstrongCardaliaguetSouganidis14}, where they obtained quantitative results for stochastic homogenization. We also mention~\cite{Viterbo25}, in which Viterbo provided an extension of symplectic homogenization to the stationary ergodic setting.

    The question of the homogenization of Hamilton--Jacobi equations in the general case where $H$ is not convex in $p$ remained opened for a while. Although a few particular cases were investigated ({\it e.g.}, \cite{ArmstrongSouganidis13,ArmstrongTranYu15, ArmstrongTranYu16, Gao16, Gao19}), the first counterexample was provided by Ziliotto in~\cite{Ziliotto17}, proving that in the stationary ergodic setting the homogenization of Hamilton--Jacobi equations generally fails without the convexity assumption.

    \bigskip

    \subsection{Organization of the article}

    This article is organized as follows.
    \begin{itemize}
        \item In Section~\ref{sec:2}, we describe our setting, in particular the parameter space $\Omega$ (Section~\ref{sec:2.1}) and the ambient manifold (Section~\ref{sec:2.2}). In Section~\ref{sec:2.3}, we introduce the notion of family of stationary-ergodic Hamiltonians, list the needed assumptions, and state our homogenization result (Main Theorem).
        \item Section~\ref{sec:3} is devoted to the study of the asymptotic geometry of the ambient manifold: we recall the needed notions and show how the rescaled spaces converge to the asymptotic cone of the group $\G$ (Section~\ref{conesec}). Moreover, in Section~\ref{sec:3.2} we introduce a notion of convergence for functions from the rescaled metric spaces to the asymptotic cone.
        \item In Section~\ref{effLsec}, we define the effective Lagrangian, which plays a fundamental role in both the proof of convergence and the determination of the effective Hamiltonian (which corresponds to its convex conjugate). This function is obtained via action-minimizing methods, inspired by Aubry--Mather theory, combined with subadditive ergodic techniques.
        \item The proof of the main convergence result (Main Theorem) is contained in Section~\ref{sec:5}.
        \item The appendices provide a description of several examples that motivate the entire framework (Appendix~\ref{sec:exe}), a discussion on how to use our main result to introduce a notion of stable-like norm adapted to stationary-ergodic families of Riemannian metrics (Appendix~\ref{sec:stablenorm}), as well as some auxiliary technical results (Appendix~\ref{app:A}).
    \end{itemize}

    \subsection*{Acknowledgments}

    The authors are grateful to Andrea Davini for his insightful comments. The authors acknowledge the support of the Italian Ministry of University and Research's PRIN 2022 grant ``\textit{Stability in Hamiltonian dynamics and beyond}'' and of the Department of Excellence grant MatMod@TOV (2023-27), awarded to the Department of Mathematics at University of Rome Tor Vergata. The authors are members of the INdAM research group GNAMPA. AS is member the UMI group DinAmicI.

    \bigskip

    \section{Setting and Main Result}\label{sec:2}

    \subsection{Probability Space}\label{sec:2.1}

    Throughout this paper $(\Omega,\Fcal,\Pds)$ will denote a probability space, where $\Pds$ is a probability measure and $\Fcal$ is the $\sigma$-algebra of the $\Pds$-measurable sets. We will say that a property holds \emph{almost surely} (a.s.\ for short) if it is valid up to a set of {$\Pds$-measure (or just {\it probability})} 0. A set of probability 1 will be called a set of \emph{full measure}.

    Let $\Xcal$ be a $\sigma$-algebra over a set $X$ and $\Xcal'$ be a $\sigma$-algebra over a set $X'$. A function $f:X\to X'$ will be called measurable if $f^{-1}(E)\in\Xcal$ whenever $E\in\Xcal'$. A topological space $X$ will always be assumed endowed with its Borel $\sigma$-algebra $\Bcal(X)$. Be \emph{random variable}, we will usually refer to a measurable function as above with $(X,\Xcal)=(\Omega,\Fcal)$.

    We are particularly interested in the case where the range of a random variable is a \emph{Polish space}, {\it i.e.}, a complete and separable space. By \emph{continuous random function}, we will intend a random variable whose image is in the Polish space of the continuous functions endowed with the metric of uniform convergence on compact subsets. With straightforward modification we retrieve this result from~\cite{DaviniSiconolfi09}:

    \begin{prop}
        Let $M$ be a manifold and $\omega\mapsto v(\cdot,\omega)$ be a map from $\Omega$ to $\Crm(M)$. The following are equivalent facts:
        \begin{myenum}
            \item $v$ is a random continuous function;
            \item $\omega\mapsto v(x,\omega)$ is measurable for every $x\in M$;
            \item the map $(x,\omega)\mapsto v(x,\omega)$ is jointly measurable, {\it i.e.}, measurable with respect to the product $\sigma$-algebra $\Bcal(M)\otimes\Fcal$.
        \end{myenum}
    \end{prop}

    \begin{defin}\label{def:induceddynG}
        Let $\G$ be a group with a given topology. We define a \emph{dynamical system induced by} $\G$ as a family of mappings of $\Omega$ into itself, denoted by $\{\tau_g\}_{g\in \G}$, which satisfies the following conditions:
        \begin{myenum}
            \item the \emph{group property}: $\tau_{gg'}=\tau_g\circ\tau_{g'}$ for any $g,g'\in \G$;
            \item the mappings $\tau_g:\Omega\to\Omega$ are measurable and \emph{measure preserving}, {\it i.e.}, $\Pds(\tau_g E)=\Pds(E)$ for all $E\in\Fcal$;
            \item the map $(g,\omega)\to\tau_g\omega$ from $\G\times\Omega$ to $\Omega$ is jointly measurable, {\it i.e.}, measurable with respect to the $\sigma$-algebras $\Bcal(\G)\otimes\Fcal$, $\Fcal$.
        \end{myenum}
        We will say that $\{\tau_g\}_{g\in \G}$ is \emph{ergodic} if it satisfies one of the following equivalent conditions:
        \begin{mylist}
            \item every measurable function $f$ defined on $\Omega$ such that $f(\tau_g\omega)=f(\omega)$ a.s.\ for any $g\in \G$, is almost surely constant;
            \item if $E\in\Fcal$ is such that $\Pds(\tau_g E\,\Delta\,E)=0$ for all $g\in \G$, where $\Delta$ stands for the symmetric difference, then $E$ has probability either 1 or 0.
        \end{mylist}
        Given a random variable $f:\Omega\to\Rds$, for any fixed $\omega\in\Omega$ the map $g\mapsto f(\tau_g\omega)$ is called a \emph{realization} of $f$.
    \end{defin}

    \subsection{The Ambient Space}\label{sec:2.2}

    Hereafter $M$ will denote a smooth connected finite-dimensional manifold without boundary, endowed with a complete Riemannian metric. We denote by $|\cdot|_x$ the induced norm on either $T_x M$ and $T^*_x M$, the fibers of respectively the tangent bundle $TM$ and the cotangent bundle $T^*M$. Finally, $d$ denotes the corresponding Riemannian distance on $M$.

    We consider an abelian group $\G$ acting on $M$ by isometries, namely,
    \begin{mylist}
        \item each $g\in \G$ corresponds to an isometry that we denote (with an abuse of notation) $g: M \longrightarrow M$;
        \item $g_1g_2(x)=g_1(g_2(x))$ for any $g_1,g_2\in \G$ and $x\in M$;
        \item there is an $e\in \G$ such that $e(x)=x$ for all $x\in M$ (here $e$ is the identity of $\G$).
    \end{mylist}

    We require the action induced by $\G$ to be:
    \begin{enumerate}[label=\textbf{(G\arabic*)},ref=\textnormal{\textbf{(G\arabic*)}},font=\normalfont]
        \item\label{Gfree} \emph{Free}: if $g(x)=x$ for some $x\in M$, then $g=e$;
        \item\label{Gtdisc} \emph{Totally discontinuous}: every point $x\in M$ has a neighborhood $U$ such that $g(U)\cap U=\emptyset$ for any $g\in \G$ with $g(x)\ne x$;
        \item\label{Gcocomp} \emph{Co-compact}: the quotient space $M/\G$ is compact with respect to the quotient topology.
    \end{enumerate}

    \medskip

    \begin{rem}\label{thisiscov}
        We point out that, under these conditions, the projection
        \begin{equation*}
            \pi:M\longrightarrow M/\G
        \end{equation*}
        induces a differential structure for which $M/\G$ is a compact Riemannian manifold and $\pi$ is a local isometry.
        \begin{myenum}[wide=0pt]
            \item\label{en:thisiscov1} It is known, see~\cite[Proposition~3.4.15]{BuragoBuragoIvanov01}, that the projection $\pi:M\to M/\G$ is a covering map and $\G$ is its group of deck transformations. In particular, this covering is \emph{regular}, {\it i.e.}, for every $y\in M/\G$ and $x,x'\in\pi^{-1}(y)$ there is a unique deck transformation $g$ so that $g(x)=x'$. It is easy to see that asking {the action of $\G$} to be free and totally discontinuous is equivalent to require that each $x\in M$ has a neighborhood $U$ so that
            \begin{equation*}
                g(U)\cap g'(U)=\emptyset,\qquad \text{for any $g,g'\in \G$ with $g\ne g'$}.
            \end{equation*}
            An action satisfying this property is sometimes called a \emph{covering space action}.\\
            Standard results in algebraic topology further show that $\G$ is a finitely generated abelian group, therefore
            \begin{equation}\label{eq:thisiscov1}
                \G\simeq\Zds^b\times\Tcal,
            \end{equation}
            where $b\ge0$ is the rank of $\G$ and $\Tcal$ is a (finite) torsion group. With an abuse of notation we will identify {both $\Zds^0$ and $\Rds^0$ with the set $\{0\}$}. Let us observe that if $b$ in~\eqref{eq:thisiscov1} is $0$, namely $\G$ is a torsion group, and only in this case, we have that $M$ is compact, see \zcref{Gtorsion}.
            \item It follows from the above discussion that $M$ corresponds to an abelian cover of $\widehat M\coloneqq M/\G$. Vice versa, any abelian cover fits to our setting. We recall that an \emph{abelian cover} of a compact, connected smooth manifold $ \widehat M $ is a covering space $ \pi: M \to \widehat M $ where the covering group is abelian. In other words, the covering space $ M $ has a structure such that the group of deck transformations (which is the group of homeomorphisms of $ M $ that commute with the covering map $ \pi $) forms an abelian group. A \emph{maximal abelian cover} is the covering space where the abelian group of deck transformations is as large as possible, meaning that it is the largest abelian cover that can be obtained for the manifold $ \widehat M $ (sometimes it is also called the \emph{universal Abelian covering space}). We remark that in our case we are not assuming that $M$ corresponds to a maximal abelian cover.
        \end{myenum}
    \end{rem}

    \begin{rem}\label{actonbundle}
        $\G$ also defines a free and totally discontinuous group of actions on $TM$ and $T^*M$ as well. This is done by setting for each $g\in \G$
        \begin{align*}
            g_*:TM&\longrightarrow TM,&g^*:T^*M&\longrightarrow T^*M,\\
            (x,q)&\longmapsto\mleft(g(x),\drm g_x(q)\mright),&(x,p)&\longmapsto\mleft(g(x),p \circ (\drm g_x)^{-1}\mright).
        \end{align*}
    \end{rem}

    \subsection{The Hamiltonian, the problem and the Main Theorem}\label{sec:2.3}

    Throughout the rest of the paper we will implicitly assume the following standing assumptions:\\

    \noindent{\bf Standing assumptions:} {\it Let $M$ be a smooth connected finite-dimensional manifold without boundary, endowed with a complete Riemannian distance $d$ and let $(\Omega,\Fcal,\Pds)$ be a probability space.

        \smallskip

        Let $\G$ be a finitely generated abelian group of rank $b$ such that:
        \begin{itemize}
            \item $\G$ acts on $M$ by isometries and that this action satisfies hypotheses \zcref[comp,noname]{Gfree,Gtdisc,Gcocomp};
            \item $\G$ induces a dynamical system $\{\tau_g\}_{g\in \G}$ on $\Omega$, as in Definition \ref{def:induceddynG}.
        \end{itemize}
    }

    \medskip

    We are now interested in a stochastic Hamilton--Jacobi equation posed on $M$. 

    \begin{defin}\label{statdef}
        A jointly measurable function $v$ defined in $M\times\Omega$ is said \emph{stationary} if for each $g\in \G$ there is a set $\Omega_g$ of full measure such that
        \begin{equation*}
            v(x,\tau_g\omega)=v(g(x),\omega),\qquad \text{for every $x\in M$ and $\omega\in\Omega_g$}.
        \end{equation*}
    \end{defin}

    Taking into account \zcref{actonbundle}, \zcref{statdef} can be seamlessly extended to jointly measurable functions defined in $TM\times\Omega$.

    Next we consider a Hamiltonian $H:T^*M\times\Omega\to\Rds$ satisfying the following conditions:
    \begin{Hassum}
        \item\label{condcont} $\omega\mapsto H(\cdot,\cdot,\omega)$ is a continuous random function;
        \item\label{condreg} for each $\omega\in\Omega$, $H(\cdot,\cdot,\omega)$ is $\Crm^2$;
        \item\label{condconv} for any $(x,\omega)\in M\times\Omega$ the map $p\mapsto H(x,p,\omega)$ is strictly convex;
        \item\label{condcoerc} there exist two convex nondecreasing superlinearly coercive functions $\vartheta,\Theta:\Rds^+\to\Rds$ such that
        \begin{equation*}
            \vartheta(|p|_x)\le H(x,p,\omega)\le\Theta(|p|_x),\qquad \text{for any }(x,p,\omega)\in T^*M\times\Omega;
        \end{equation*}
        \item\label{condstat} $H$ is stationary, {\it i.e.},
        \begin{equation*}
            H(x,p,\tau_g\omega)=H(g^*(x,p),\omega),\qquad \text{for any $(x,p,\omega)\in T^*M\times\Omega$ and $g\in \G$}.
        \end{equation*}
    \end{Hassum}

    Moreover, $H$ is the realization of a Hamiltonian defined on the quotient manifold $M/\G$. Indeed, if we fix for each $y\in M/\G$ an $x_y\in\pi^{-1}(y)$ so that $x_y\mapsto y$ is a local isometry and define
    \begin{equation}\label{eq:Hproj}
        \whH(y,p,\omega)\coloneqq H(x_y,p,\omega),\qquad \text{for each }(x,p,\omega)\in T^*(M/\G)\times\Omega,
    \end{equation}
    we have that $\whH$ is a Hamiltonian on $T^*(M/\G)\times\Omega$ satisfying \zcref[comp]{condcont,condconv,condreg,condcoerc} and
    \begin{equation}\label{eq:lift2rlztn}
        \whH(y,p,\tau_g\omega)=H(x_y,p,\tau_g\omega)=H(g^*(x_y,p),\omega),
    \end{equation}
    which is to say that $g\mapsto H(g^*(x_y,p),\omega)$ is a realization of $\omega\mapsto\whH(y,p,\omega)$. Conversely, given a Hamiltonian $\whH$ on $T^*(M/\G)\times\Omega$ its realizations define a stationary Hamiltonian on $M$.

    \medskip

    \begin{rem}
        We point out that if $\Omega$ is a singleton, then our problem correspond to the periodic case studied in~\cite{ContrerasIturriagaSiconolfi14}, where the Hamiltonian $H$ defined on $T^*M$ is deterministic and corresponds to the \emph{lift} of a Hamiltonian $\whH$ defined in $T^*(M/\G)$, namely
        \begin{equation*}
            H(x,p,\omega)=\whH(\pi^*(x,p),\omega),\qquad \text{for any $(x,p)\in T^*M$},
        \end{equation*}
        where $\pi^*$ is the inverse of the pullback by the projection $\pi:M\to M/\G$. See also \zcref{sec:exe}.
    \end{rem}

    \begin{rem}\label{liftisdet}
        If $H$ is the lift of a Hamiltonian $\whH:T^*(M/\G)\times\Omega\to\Rds$, then, for any $(x,p)\in T^*M$, $\omega\in\Omega$ and $g\in \G$,
        \begin{equation*}
            \whH(\pi^*(x,p),\tau_g\omega)=H(x,p,\tau_g\omega)=H(g^*(x,p),\omega)=\whH(\pi^*(x,p),\omega).
        \end{equation*}
        The ergodicity of $\{\tau_g\}$ thus yields that $\whH$, and accordingly $H$, is a.s.\ deterministic.
    \end{rem}

    We refer to Appendix~\ref{sec:exe} for the discussion of more motivating examples fitting our framework.

    \bigskip

    For any $\varepsilon>0$ we will denote with $M_\varepsilon$ the metric space $(M,d_\varepsilon\coloneqq\varepsilon d)$ and define
    \begin{equation*}
        H_\varepsilon(x,p,\omega)\coloneqq H\mleft(x,\dfrac p\varepsilon,\omega\mright),\qquad \text{for each }(x,p,\omega)\in T^*M\times\Omega.
    \end{equation*}
    We further set the Lagrangians $L$ and $L_\varepsilon$ as the convex conjugates of $H$ and $H_\varepsilon$, respectively. We point out that
    \begin{equation}\label{eq:lagequiv}
        L_\varepsilon(x,q,\omega)=L(x,\varepsilon q,\omega).
    \end{equation}

    We are interested in the asymptotic behavior, as $\varepsilon$ tends to 0, of the solution to the following stochastic Hamilton--Jacobi equation:
    \begin{equation}\label{eq:HJeps}\tag{H\textsubscript{\ensuremath{\varepsilon}}J\ensuremath{\omega}}
        \mleft\{
        \begin{aligned}
            &\partial_t u(x,t,\omega)+H_\varepsilon\mleft(x,\partial_x u,\omega\mright)=0,&& \text{in }M_\varepsilon\times(0,\infty)\times\Omega,\\
            &u(x,0,\omega)=u^\varepsilon_0(x,\omega),&& \text{in }M_\varepsilon\times\Omega,
        \end{aligned}
        \mright.
    \end{equation}
    where $x\mapsto u_0^\varepsilon(x,\omega)$ is a random uniformly continuous function for any $\omega\in\Omega$.

    \medskip

    As proved in~\cite[Theorems~1.1, 1.2 and Corollary~8.12]{Fathi24}, for each $\varepsilon>0$ and $\omega\in\Omega$ fixed, the unique uniformly continuous solution to~\eqref{eq:HJeps} is given by the value function
    \begin{equation}\label{eq:vfeps}
        u_\varepsilon(x,t,\omega)\coloneqq\inf_{x'\in M_\varepsilon}\mleft\{u_0^\varepsilon\mleft(x',\omega\mright)+\phi_\varepsilon\mleft(x',x,t,\omega\mright)\mright\},
    \end{equation}
    where $\phi_\varepsilon$ is defined by the variational formula
    \begin{equation}\label{eq:minacteps}
        \phi_\varepsilon\mleft(x',x,t,\omega\mright)\coloneqq\inf\mleft\{\int_0^t L_\varepsilon(\gamma,\dot\gamma,\omega)d\tau\mright\}
    \end{equation}
    and the infimum above is taken over the absolutely continuous curves $\gamma:[0,t]\to M_\varepsilon$ with $\gamma(0)=x'$ and $\gamma(t)=x$.

    \begin{rem}\label{phistat}
        We report the following properties of the minimal action $\phi_\varepsilon$.
        \begin{myenum}[wide=0pt]
            \item Setting
            \begin{equation}\label{eq:minact}
                \phi\mleft(x',x,t,\omega\mright)\coloneqq\inf\mleft\{\int_0^t L(\gamma,\dot\gamma,\omega)d\tau\mright\},
            \end{equation}
            where the infimum above is taken over the absolutely continuous curves $\gamma:[0,t]\to M$ with $\gamma(0)=x'$ and $\gamma(t)=x$, \zcref{eq:lagequiv} yields
            \begin{equation}\label{eq:minactequiv}
                \phi_\varepsilon\mleft(x',x,t,\omega\mright)=\varepsilon\phi\mleft(x',x,\frac t\varepsilon,\omega\mright),\qquad \text{for all $\mleft(x',x,t\mright)\in M^2\times\Rds^+$, $\omega\in\Omega$ and $\varepsilon>0$}.
            \end{equation}
            \item\label{en:phistat2} Since $H$ is stationary by~\ref{condstat}, so is $L$. It then follows that, for any $(x',x,t)\in M_\varepsilon^2\times\Rds^+$, $\omega\in\Omega$, $\varepsilon>0$ and $g\in \G$,
            \begin{equation*}
                \phi_\varepsilon\mleft(x',x,t,\tau_g\omega\mright)=\phi_\varepsilon\mleft(g\mleft(x'\mright),g(x),t,\omega\mright).
            \end{equation*}
        \end{myenum}
    \end{rem}

    \medskip

    We can now state our Main Theorem.

    \begin{main}
        \noindent Given a Hamiltonian $H:T^*M\times\Omega\to\Rds$ satisfying \zcref[comp]{condcoerc,condcont,condconv,condreg,condstat}, there exists a convex and superlinear effective Hamiltonian $\ovH:\Rds^b\to\Rds$ such that the following holds.

        \smallskip

        \noindent Let $\{u_0^\varepsilon\}_{\varepsilon>0}$ be a collection of random functions from $M_\varepsilon=(M, \varepsilon d)$ to $\Rds$ uniformly equicontinuous with respect to the metrics $\varepsilon d $ (in the sense of \zcref{eqcontdef}) and a.s.\ pointwise convergent (in the sense of \zcref{coneconvdef}) to $v_0:\Rds^b\to\Rds$ as $\varepsilon$ goes to $0^+$. Then, the solution $u_\varepsilon$ to~\zcref{eq:HJeps} locally uniformly converges (in the sense of \zcref{coneconvdef}) to the only uniformly continuous solution to
        \begin{equation*}
            \mleft\{
            \begin{aligned}
                &\partial_t v(h,t)+\ovH(\partial_h v)=0,&& \text{in $\Rds^b\times(0,\infty)$},\\
                &v(h,0)=v_0(h),&& \text{in $\Rds^b$}.
            \end{aligned}
            \mright.
        \end{equation*}
    \end{main}

    \medskip

    The proof of the Main Theorem will be provided in Section~\ref{sec:5}. 

    \bigskip

    \section{Asymptotic Geometry}\label{sec:3}

    \subsection{The Asymptotic Cone}\label{conesec}

    We start this \zcref[noref,nocap]{conesec} by recalling some definitions.

    \begin{defin}
        Given two metric spaces $(X,d_X)$ and $(Y,d_Y)$ and a map $f:X\to Y$, we define the \emph{distortion} of $f$ as
        \begin{equation*}
            \dis f\coloneqq\sup_{x,x'\in X}\mleft|d_X\mleft(x,x'\mright)-d_Y\mleft(f(x),f\mleft(x'\mright)\mright)\mright|.
        \end{equation*}
        If there exists an $\epsilon>0$ such that for any $y\in Y$ there is an $x\in X$ so that $d_Y(y,f(x))\le\epsilon$ and $\dis f\le\epsilon$, we will say that $f$ is an \emph{$\epsilon$-isometry}.
    \end{defin}

    \begin{rem}\label{compepsiso}
        It is apparent that the composition of an $\epsilon_1$-isometry and an $\epsilon_2$-isometry is an \emph{$(\epsilon_1+\epsilon_2)$-isometry}.
    \end{rem}

    \begin{defin}
        A \emph{pointed metric space} is a pair $(X,x)$ consisting of a metric space $X$ and a reference point $x\in X$.\\
        We will say that a sequence of pointed metric spaces $\{(X_n,x_n)\}_{n\in\Nds}$ \emph{converges in the Gromov--Hausdorff sense} to a pointed metric space $(X,x)$ if, for any $\delta>0$, $r>0$ and $n\in\Nds$ big enough, there is a map
        \begin{equation*}
            f_n:B_n(x_n,r)\longrightarrow X,
        \end{equation*}
        where $B_n(x_n,r)\subseteq X_n$ is the ball centered at $x_n$ with radius $r$, such that
        \begin{mylist}
            \item $f_n(x_n)=x$;
            \item $\dis f_n<\delta$;
            \item the $\delta$-neighborhood of $f_n(B_n(x_n,r))$ contains the ball $B_{r-\delta}(x)\subseteq X$.
        \end{mylist}
    \end{defin}

    \begin{rem}
        It follows from the definition that if a sequence of pointed metric spaces converges to $(X,x)$ it also converges to its completion, thus we will always assume that the Gromov--Hausdorff limit spaces are complete.\\
        It is shown in~\cite[Theorem~8.1.7]{BuragoBuragoIvanov01} that the Gromov--Hausdorff limit of a sequence of pointed spaces is essentially unique: if two complete pointed metric spaces $(X,x)$ and $(X',x')$ are the Gromov--Hausdorff limits of the same sequence and all the closed bounded subset of $X$ are compact, then there is an isometry $f:X\to X'$ such that $f(x)=x'$.
    \end{rem}

    \begin{defin}\label{asconedef}
        Let $(X,d)$ be a metric space and define, for any $\varepsilon>0$, the rescaled metric spaces $X_\varepsilon\coloneqq(X,\varepsilon d)$. We call the \emph{Gromov--Hausdorff asymptotic cone} of $X$ the limit space in the Gromov--Hausdorff sense, if it exists, of the pointed metric spaces $\{(X_\varepsilon,x)\}$ as $\varepsilon$ tends to 0. It can be proved (\cite[Proposition~8.2.8]{BuragoBuragoIvanov01}) that the asymptotic cone does not depend on the choice of the reference point $x$.
    \end{defin}

    \begin{prop}\label{epsisoghcone}
        Let $(X,d_X)$ and $(Y,d_Y)$ be two metric spaces, and assume that there is an $\epsilon$-isometry between them. If $X$ has an asymptotic cone then it is also an asymptotic cone of $Y$.
    \end{prop}

    The previous \zcref[noref]{epsisoghcone}, which follows from~\cite{BuragoBuragoIvanov01}, is crucial for our study. Indeed, to provide an asymptotic cone of the manifold $M$ we will focus on a much simpler space which is $\epsilon$-isometric to $M$.

    First we fix $z\in M$ and define the following distance on $\G$:
    \begin{equation}\label{eq:orbdist}
        d_z\mleft(g,g'\mright)\coloneqq d\mleft(g(z),g'(z)\mright),\qquad \text{for $g,g'\in \G$}.
    \end{equation}
    We will refer to distances defined as in~\eqref{eq:orbdist} as \emph{orbit metrics}. Each element of $\G$ is an isometry, hence orbit metrics are \emph{invariant}, {\it i.e.}, $d_z(g,g')=d_z\mleft(0,g^{-1}g'\mright)$.

    \begin{lem}\label{MepsisoG}
        For all $z\in M$, the map
        \begin{align*}
            f_z:(\G,d_z)&\,\longrightarrow(M,d)\\
            g&\,\longmapsto g(z),
        \end{align*}
        is an $\epsilon$-isometry between $(M,d)$ and $(\G,d_z)$.
    \end{lem}
    \begin{proof}
        Setting
        \begin{equation*}
            \epsilon\coloneqq\sup_{x,x'\in M}\inf_{g\in \G}d\mleft(x,g\mleft(x'\mright)\mright)+1\le\diam(M/\G)+1,
        \end{equation*}
        we get that for any $x\in M$ there is a $g\in \G$ so that
        \begin{equation*}
            d(f_z(g),x)=d(g(z),x)<\epsilon,
        \end{equation*}
        while
        \begin{equation*}
            \dis f_z=\max_{g,g'\in \G}\mleft|d\mleft(f_z(g),f_z\mleft(g'\mright)\mright)-d_z\mleft(g,g'\mright)\mright|=0<\epsilon.
        \end{equation*}
        These inequalities show that $f$ is an $\epsilon$-isometry.
    \end{proof}

    We recall that, see \zcref{thisiscov}\ref{en:thisiscov1}, $\G$ is isomorphic to $\Zds^b\times\Tcal$, where $b$ is the rank of $\G$ and $\Tcal$ is a torsion group. Since $\Tcal$ is finite, the next \zcref[noref]{projepsiso} is apparent.

    \begin{lem}\label{projepsiso}
        Let $b$ be the rank of $\G$ and $d_z$ be an orbit metric on $\G$. Then the projection
        \begin{equation*}
            \G\simeq\Zds^b\times\Tcal\longrightarrow\Zds^b
        \end{equation*}
        is (for some $\epsilon>0$) an $\epsilon$-isometry between $(\G,d_z)$ and $\mleft(\Zds^b,\ovd_z\mright)$, where $\ovd_z$ is the metric on $\Zds^b$ induced by this projection.
    \end{lem}

    Moreover:

    \begin{lem}\label{equivorbit}
        Given $z\in M$, $\ovd_z$ is bi-Lipschitz equivalent to the Euclidean norm $|\cdot|$ on $\Zds^b$, namely there is a constant $A>0$ so that
        \begin{equation}\label{eq:equivorbit.1}
            A^{-1}\mleft|h-h'\mright|\le\ovd_z\mleft(h,h'\mright)\le A\mleft|h-h'\mright|,\qquad \text{for every $h,h'\in\Zds^b$}.
        \end{equation}
    \end{lem}
    \begin{proof}
        If $b=0$ our claim is trivial, thus we assume that $b>0$.\\
        We start observing that $\G/\Tcal$, and hence $\Zds^b$, acts on $M$ by isometries and is free. Moreover, we have that $\Tcal$ acts by isometries on $M/(\G/\Tcal)$ and
        \begin{equation*}
            M/\G=(M/(\G/\Tcal))/\Tcal
        \end{equation*}
        is compact, thus $\Tcal$ is free, totally discontinuous and co-compact. It follows from \zcref{Gtorsion} that $M/(\G/\Tcal)$ is compact and consequently $\G/\Tcal$ is co-compact. Finally, thanks to~\cite[Theorem~8.3.19]{BuragoBuragoIvanov01}, we get that every orbit metric $\ovd_z$ on $\G/\Tcal$ is bi-Lipschitz equivalent to the Euclidean norm, thus proving~\eqref{eq:equivorbit.1}.
    \end{proof}

    Clearly $\ovd_z$ is an invariant metric on $\Zds^b$, hence~\cite[Proposition~8.5.3]{BuragoBuragoIvanov01} yields that there is a unique continuous seminorm $\|\cdot\|$ on $\Rds^b$ such that
    \begin{equation}\label{eq:orbdist2snorm}
        \lim_{n\to\infty}\frac{\ovd_z(0,nh)}n=\|h\|,\qquad \text{for any }h\in\Zds^b.
    \end{equation}

    \begin{prop}
        Given $z\in M$, $\|\cdot\|$ is a norm on $\Rds^b$.
    \end{prop}
    \begin{proof}
        It is enough to show that
        \begin{equation}\label{eq:stabnorm1}
            A^{-1}|h|\le\|h\|\le A|h|,\qquad \text{for any $h\in\Rds^b$},
        \end{equation}
        where $|\cdot|$ is the Euclidean norm on $\Rds^b$. Let $\{h_i\}_{i=1}^b$ be a basis of $\Zds^b$ and define, for every $\varepsilon>0$, the functions $\Phi_\varepsilon:\Rds^b\to\Zds^b$ such that, for any $h\coloneqq\sum\limits_{i=1}^b a_i h_i$,
        \begin{equation}\label{eq:stabnorm2}
            \Phi_\varepsilon(h)\coloneqq\sum_{i=1}^b\mleft\lfloor\frac{a_i}\varepsilon\mright\rfloor h_i.
        \end{equation}
        It is apparent that
        \begin{equation}\label{eq:stabnorm3}
            \lim_{\varepsilon\to0^+}\varepsilon \Phi_\varepsilon(h)=h,\qquad \text{for all $h\in\Rds^b$}.
        \end{equation}
        We know from \zcref{equivorbit} that, for some $A>0$,
        \begin{equation*}
            A^{-1}|h|\le\ovd_z(0,h)\le A|h|,\qquad \text{for any $h\in\Zds^b$},
        \end{equation*}
        thus, by~\eqref{eq:orbdist2snorm},
        \begin{equation*}
            A^{-1}|h|\le\|h\|\le A|h|,\qquad \text{for any $h\in\Zds^b$},
        \end{equation*}
        which in turn shows that
        \begin{equation*}
            A^{-1}|\Phi_\varepsilon(h)|\le\|\Phi_\varepsilon(h)\|\le A|\Phi_\varepsilon(h)|,\qquad \text{for every $h\in\Rds^b$ and $\varepsilon>0$}.
        \end{equation*}
        This, \zcref{eq:stabnorm3} and the continuity of the seminorm prove~\eqref{eq:stabnorm1}.
    \end{proof}

    \begin{rem}\label{allnormequiv}
        The norm $\|\cdot\|$ is usually called the \emph{stable norm} of the metric $\ovd_z$. However, all norms are equivalent on finite dimensional vector spaces and we do not need to refer to a specific norm in our analysis, therefore, for simplicity, we can assume that $\|\cdot\|$ is the Euclidean norm.
    \end{rem}

    The following \zcref[noref]{conemac} is a straightforward consequence of~\cite[Theorem~8.5.4]{BuragoBuragoIvanov01}, \zcref{MepsisoG,projepsiso,compepsiso,epsisoghcone}.

    \begin{theo}\label{conemac}
        Let $b$ be the rank of $\G$, then $\mleft(\Rds^b,\|\cdot\|\mright)$ is the asymptotic cone of $M$.
    \end{theo}

    Notice that, consistently with \zcref{asconedef,Gtorsion}, if $M$ is compact then its asymptotic cone is just a point.

    Henceforth $\mleft(\Rds^b,\|\cdot\|\mright)$ will always denote the asymptotic cone of $M$.

    \subsection{Convergence of Functions}\label{sec:3.2}

    We need to introduce a notion of convergence of functions from the rescaled metric spaces $M_\varepsilon$ to the asymptotic cone $\Rds^b$. To do so we will use the identification $\G\simeq\Zds^b\times\Tcal$ and a map $\Phi_\varepsilon:\Rds^b\to\Zds^b$ so that
    \begin{equation}\label{eq:phiepsconv}
        \lim_{\varepsilon\to0^+}\|\varepsilon\Phi_\varepsilon(h)-h\|\le C_\varepsilon,\qquad \text{for any $h\in\Rds^b$},
    \end{equation}
    where $\{C_\varepsilon\}_{\varepsilon>0}$ is an infinitesimal sequence of positive numbers, see for instance~\eqref{eq:stabnorm2}. More in details, when convenient every $g\in \G$ will be represented as a $(h,k)\in\Zds^b\times\Tcal$, while $\Phi_\varepsilon$ will be used to relate $\Rds^b$ to $\G$.

    \begin{defin}\label{coneconvdef}
        Given a function $f:\Rds^b\to\Rds$ and, for every $\varepsilon>0$, the maps $f_\varepsilon:M_\varepsilon\to\Rds$ we will say that
        \begin{mylist}
            \item $f_\varepsilon$ \emph{$\varepsilon\Phi_\varepsilon$-converges pointwise} to $f$ if
            \begin{equation*}
                \lim_{\varepsilon\to0^+}f_\varepsilon((\Phi_\varepsilon(h),k)(x))=f(h),\qquad \text{for any $h\in\Rds^b$, $k\in\Tcal$ and $x\in M_\varepsilon$};
            \end{equation*}
            \item $f_\varepsilon$ \emph{$\varepsilon\Phi_\varepsilon$-converges locally uniformly} to $f$ if, for each compact subset $K$ of $\Rds^b$,
            \begin{equation*}
                \lim_{\varepsilon\to0^+}\sup_{h\in K}|f_\varepsilon((\Phi_\varepsilon(h),k)(x))-f(h)|=0,\qquad \text{for any $k\in\Tcal$ and $x\in M_\varepsilon$}.
            \end{equation*}
        \end{mylist}
    \end{defin}

    As usual for this kind of problems, it is useful to rely on some sort of equicontinuity notion.

    \begin{defin}\label{eqcontdef}
        Given the functions $f_\varepsilon:M_\varepsilon\to\Rds$ we will say that
        \begin{mylist}
            \item the functions $f_\varepsilon$ are \emph{$d_\varepsilon$-equicontinuous} if for each $x_0\in M$, $r>0$ and $\epsilon>0$ there exists a $\delta>0$ such that
            \begin{equation*}
                \mleft|f_\varepsilon(x)-f_\varepsilon\mleft(x'\mright)\mright|<\epsilon,\qquad \text{for any $\varepsilon>0$ and $x,x'\in B_\varepsilon(x_0,r)\subseteq M_\varepsilon$ with $d_\varepsilon(x,x')<\delta$};
            \end{equation*}
            \item the functions $f_\varepsilon$ are \emph{uniformly $d_\varepsilon$-equicontinuous} if for each $\epsilon>0$ there exists a $\delta>0$ such that
            \begin{equation*}
                \mleft|f_\varepsilon(x)-f_\varepsilon\mleft(x'\mright)\mright|<\epsilon,\qquad \text{for any $\varepsilon>0$ and $x,x'\in M_\varepsilon$ with $d_\varepsilon(x,x')<\delta$};
            \end{equation*}
            \item the functions $f_\varepsilon$ are \emph{$d_\varepsilon$-equiLipschitz continuous} if there exists an $\ell>0$ such that
            \begin{equation*}
                \mleft|f_\varepsilon(x)-f_\varepsilon\mleft(x'\mright)\mright|\le\ell d_\varepsilon\mleft(x,x'\mright),\qquad \text{for any $\varepsilon>0$ and $x,x'\in M_\varepsilon$};
            \end{equation*}
            \item the functions $f_\varepsilon$ are \emph{locally $d_\varepsilon$-equiLipschitz continuous} if for each $x_0\in M$ and $r>0$ there exists an $\ell>0$ such that
            \begin{equation*}
                \mleft|f_\varepsilon(x)-f_\varepsilon\mleft(x'\mright)\mright|\le\ell d_\varepsilon\mleft(x,x'\mright),\qquad \text{for any $\varepsilon>0$ and $x,x'\in B_\varepsilon(x_0,r)\subseteq M_\varepsilon$}.
            \end{equation*}
        \end{mylist}
    \end{defin}

    We close this \zcref[noref,nocap]{conesec} with some useful results about the relationship between local equicontinuity and convergence in the sense of \zcref{coneconvdef,eqcontdef}.

    \begin{lem}\label{eqcont}
        Let $f_\varepsilon:M_\varepsilon\to\Rds$ be a $d_\varepsilon$-equicontinuous sequence of maps and $\{h_\varepsilon\}_{\varepsilon>0}\subseteq\Zds^b$ be a sequence such that, for some $h\in\Rds^b$, $\varepsilon h_\varepsilon\to h$ as $\varepsilon\to0^+$. We have that
        \begin{equation}\label{eq:eqcont.1}
            \lim_{\varepsilon\to0^+}\mleft|f_\varepsilon((\Phi_\varepsilon(h),k)(x))-f_\varepsilon\mleft(\mleft(h_\varepsilon,k'\mright)\mleft(x'\mright)\mright)\mright|=0,\qquad \text{for any $k,k'\in\Tcal$ and $x,x'\in M$}.
        \end{equation}
    \end{lem}
    \begin{proof}
        \zcref[S]{normconebound}, yields that there is a constant $A$ such that
        \begin{equation*}
            d_\varepsilon\mleft((\Phi_\varepsilon(h),k)(x),\mleft(h_\varepsilon,k'\mright)\mleft(x'\mright)\mright)\le\varepsilon A(\|\Phi_\varepsilon(h)-h_\varepsilon\|+1),\qquad \text{for any $k,k'\in\Tcal$ and $x,x'\in M$}.
        \end{equation*}
        The quantity on the right-hand side clearly converges to 0 as $\varepsilon\to0^+$, therefore, since the $f_\varepsilon$ are $d_\varepsilon$-equicontinuous, for any $n\in\Nds$ there is an $\varepsilon_n>0$ so that
        \begin{equation*}
            \mleft|f_\varepsilon((\Phi_\varepsilon(h),k)(x))-f_\varepsilon\mleft(\mleft(h_\varepsilon,k'\mright)\mleft(x'\mright)\mright)\mright|<\frac1n,\qquad \text{for any $k,k'\in\Tcal$, $x,x'\in M$ and $\varepsilon<\varepsilon_n$}.
        \end{equation*}
        This proves~\eqref{eq:eqcont.1}.
    \end{proof}

    \begin{prop}\label{equicontconvcont}
        Let $f_\varepsilon:M_\varepsilon\to\Rds$ be a $d_\varepsilon$-equicontinuous sequence of maps $\varepsilon\Phi_\varepsilon$-converging pointwise to $f:\Rds^b\to\Rds$. Then, $f$ is continuous.
    \end{prop}
    \begin{proof}
        We fix a constant $r>0$, $x_0\in M$ and a $h\in\Rds^b$ so that $\|h\|<r$, then we get from \zcref{normconebound} that, for $\varepsilon>0$ small enough, there is a constant $\ovr$ such that
        \begin{equation*}
            d_\varepsilon(x_0,(\Phi_\varepsilon(h),0)(x_0))<\ovr.
        \end{equation*}
        Given an $\epsilon>0$, the $d_\varepsilon$-equicontinuity of $\{f_\varepsilon\}$ yields that there is a $\delta_\epsilon>0$ so that
        \begin{equation}\label{eq:equicontconvcont1}
            \mleft|f_\varepsilon(x)-f_\varepsilon\mleft(x'\mright)\mright|<\epsilon,\qquad \text{for any $\varepsilon>0$ and $x,x'\in B_\varepsilon(x_0,\ovr)\subseteq M_\varepsilon$ with $d_\varepsilon(x,x')<\delta_\epsilon$}.
        \end{equation}
        Finally, \zcref{normconebound} shows that there exists $\delta>0$ so that, whenever $\|h-h'\|<\delta$,
        \begin{equation*}
            d_\varepsilon\mleft(x_0,\mleft(\Phi_\varepsilon\mleft(h'\mright),0\mright)(x_0)\mright)<\ovr,\quad \text{and}\quad d_\varepsilon\mleft((\Phi(h),0)(x_0),\mleft(\Phi_\varepsilon\mleft(h'\mright),0\mright)(x_0)\mright)<\delta_\epsilon;
        \end{equation*}
        hence, by~\eqref{eq:equicontconvcont1}, for any $h'\in\Rds^b$ with $|h-h'|<\delta$ and $\varepsilon>0$ small enough
        \begin{equation*}
            \mleft|f_\varepsilon((\Phi_\varepsilon(h),0)(x_0))-f_\varepsilon\mleft(\mleft(\Phi_\varepsilon\mleft(h'\mright),0\mright)(x_0)\mright)\mright|<\epsilon.
        \end{equation*}
        Since $\epsilon$ and $h$ are arbitrary and $f_\varepsilon$ $\varepsilon\Phi_\varepsilon$-converges pointwise to $f$, this concludes the proof.
    \end{proof}

    \begin{rem}
        In addition, arguing as in \zcref{equicontconvcont}, we get that if $\{f_\varepsilon\}$ $\varepsilon\Phi_\varepsilon$-converges pointwise to $f$ and is $d_\varepsilon$-uniformly continuous, locally $d_\varepsilon$-equiLipschitz continuous or $d_\varepsilon$-equiLipschitz continuous, then $f$ is uniformly continuous, locally Lipschitz continuous or Lipschitz continuous, respectively.
    \end{rem}

    \begin{prop}\label{equicontconvlocunif}
        Let $f_\varepsilon:M_\varepsilon\to\Rds$ be a $d_\varepsilon$-equicontinuous sequence of maps $\varepsilon\Phi_\varepsilon$-converging pointwise to $f:\Rds^b\to\Rds$. Then, $f_\varepsilon$ locally uniformly $\varepsilon\Phi_\varepsilon$-converges to $f$.
    \end{prop}
    \begin{proof}
        We fix $x\in M$, $k\in\Tcal$, a compact $K\subseteq\Rds^b$ and let $\{\varepsilon_n\}_{n\in\Nds}$ be an infinitesimal sequence and $\{h_n\}_{n\in\Nds}\subset K$ be such that
        \begin{equation}\label{eq:equicontconvlocunif1}
            \lim_{n\to\infty}|f_{\varepsilon_n}((\Phi_{\varepsilon_n}(h_n),k)(x))-f(h_n)|=\limsup_{\varepsilon\to0}\sup_{h\in K}|f_\varepsilon((\Phi_\varepsilon(h),k)(x))-f(h)|.
        \end{equation}
        Since $\{h_n\}$ is a sequence contained in a compact set we have that as $n$ goes to infinity $h_n$ tends, up to subsequences, to an $\overline h\in K$. Moreover, \zcref{eq:phiepsconv} yields
        \begin{equation}\label{eq:equicontconvlocunif2}
            \lim_{n\to\infty}\varepsilon_n\mleft\|\Phi_{\varepsilon_n}(h_n)-\Phi_{\varepsilon_n}\mleft(\overline h\mright)\mright\|\le\lim_{n\to\infty}\|\varepsilon_n\Phi_{\varepsilon_n}(h_n)-h_n\|+\mleft\|h_n-\overline h\mright\|+\mleft\|\overline h-\varepsilon_n\Phi_{\varepsilon_n}\mleft(\overline h\mright)\mright\|=0.
        \end{equation}
        Finally we observe that, according to \zcref[tlastsep={ and }]{eq:equicontconvlocunif2,eqcont,equicontconvcont},
        \begin{multline*}
            |f_{\varepsilon_n}((\Phi_{\varepsilon_n}(h_n),k)(x))-f(h_n)|\\
            \le\mleft|f_{\varepsilon_n}((\Phi_{\varepsilon_n}(h_n),k)(x))-f_{\varepsilon_n}\mleft(\mleft(\Phi_{\varepsilon_n}\mleft(\overline h\mright),k\mright)(x)\mright)\mright|+\mleft|f_{\varepsilon_n}\mleft(\mleft(\Phi_{\varepsilon_n}\mleft(\overline h\mright),k\mright)(x)\mright)-f\mleft(\overline h\mright)\mright|+\mleft|f\mleft(\overline h\mright)-f(h_n)\mright|
        \end{multline*}
        converges to $0$ as $n\to\infty$. Thanks to~\eqref{eq:equicontconvlocunif1} this proves the claim, since $x$ and $k$ have been arbitrarily chosen.
    \end{proof}

    These notions can be trivially extended to subsets or product spaces involving $M_\varepsilon$, {\it e.g.}, $f_\varepsilon:M_\varepsilon\times\Rds^+\to\Rds$ $\varepsilon\Phi_\varepsilon$-converges locally uniformly to $f:\Rds^b\times\Rds^+\to\Rds$ if, for each compact subset $K\subset\Rds^b\times\Rds^+$,
    \begin{equation*}
        \lim_{\varepsilon\to0^+}\sup_{(h,t)\in K}|f_\varepsilon((\Phi_\varepsilon(h),k)(x),t)-f(h,t)|=0,\qquad \text{for any $k\in\Tcal$ and $x\in M$}.
    \end{equation*}
    With straightforward modifications the previous results can be extended to these spaces.

    We further introduce the following \zcref[noref]{uniecontrf}, which will be useful later.

    \begin{defin}\label{uniecontrf}
        Given for all $\varepsilon>0$ the random function $f_\varepsilon:M_\varepsilon\times\Omega\to\Rds$, we will say that the $f_\varepsilon$ are \emph{uniformly $d_\varepsilon$-equicontinuous random functions} if for each $\omega\in\Omega$ and $\varepsilon>0$ the $f_\varepsilon(\omega)$ admit a common modulus of continuity.
    \end{defin}

    \section{The Effective Lagrangian}\label{effLsec}

  Our analysis relies on an asymptotic Lagrangian defined on the asymptotic cone $\Rds^b$. We define it proceeding by steps, using the minimal actions defined in \zcref{eq:minact,eq:minacteps}.

    \begin{theo}\label{effLex}
        There exists a map $\ovL:\Rds^b\to\Rds$ and a set of full measure $\ovOmega$ such that, for any $x\in M$, $k\in\Tcal$, $h\in\Rds^b$ and $\omega\in\ovOmega$,
        \begin{equation}\label{eq:effLex.1}
            \lim_{\varepsilon\to0^+}\phi_\varepsilon((0,k)(x),(\Phi_\varepsilon(h),k)(x),1,\omega)=\ovL(h).
        \end{equation}
    \end{theo}

    The map $\ovL$ is called the \emph{effective Lagrangian} and in the deterministic case corresponds to Mather's minimal action functional $\beta$, see for instance~\cite{Sorrentino15}. In the case of Riemannian metrics, it is also related to the so-called {\it stable norm}, introduced by Federer~\cite{Federer74} and popularized by Gromov~\cite{GromovLafontainePansu81,BuragoBuragoIvanov01}, see Appendix~\ref{sec:stablenorm}.

    \begin{proof}
        We start fixing $h\in\Zds^b$, $k\in\Tcal$ and $x\in M$. Then, for any $n,m\in\Nds$ with $n\ge m$ we define the random variables
        \begin{equation*}
            \phi_{m,n}(\omega)\coloneqq\phi((mh,k)(x),(nh,k)(x),n-m,\omega).
        \end{equation*}
        It is easy to see that
        \begin{myenum}
            \item $\phi_{m,n}(\omega)\le\phi_{m,l}(\omega)+\phi_{l,n}(\omega)$ for any $\omega\in\Omega$ and $m,l,n\in\Nds$ with $m\le l\le n$;
            \item by \zcref{phistat}\ref{en:phistat2} and the measure preserving property of the dynamical system induced by $\G$, the joint distributions $\{\phi_{m,n},0\le m\le n\}$ are the same as those of $\{\phi_{m+1,n+1},0\le m\le n\}$;
            \item for each $n\in\Nds$, $\Eds(|\phi_{0,n}|)<\infty$ and $\Eds(\phi_{0,n})>-nc$ for some constant $c$, as follows from \zcref{phibound}.
        \end{myenum}
        This is to say that the sequence $\{\phi_{m,n}\}$ satisfies the conditions of Kingman's subadditive ergodic Theorem (proof of a generalized version of it can be found in~\cite{Liggett85,Levental88}), therefore there is a random variable $\ovL(h,k,x)$ such that
        \begin{equation}\label{eq:effLex1}
            \lim_{n\to\infty}\frac{\phi((0,k)(x),(nh,k)(x),n,\omega)}n=\ovL(h,k,x,\omega),\qquad \text{a.s.}.
        \end{equation}
        Observe that, fixed $h\in\Rds^b$, by \zcref{normconebound,eq:phiepsconv} there is a constant $C_1>0$ and an infinitesimal sequence $\{C_\varepsilon\}_{\varepsilon>0}$ of positive numbers so that
        \begin{equation*}
            d((0,k)(x),(\Phi_\varepsilon(h),k)(x))\le\frac1\varepsilon(C_1\|h\|+C_\varepsilon),\qquad \text{for any $\varepsilon>0$},
        \end{equation*}
        {\it i.e.}, there is a constant $A_h$ so that
        \begin{equation*}
            \mleft((0,k)(x),(\Phi_\varepsilon(h),k)(x),\frac1\varepsilon\mright)\in\Ccal_{A_h},\qquad \text{for any $\varepsilon\in(0,1]$},
        \end{equation*}
        where
        \begin{equation*}
            \Ccal_{A_h}\coloneqq\mleft\{\mleft(x',x,t\mright)\in M^2\times\Rds^+:d\mleft(x',x\mright)<At\mright\}.
        \end{equation*}
        We further notice that for any $h\in\Zds^b$
        \begin{equation*}
            d((0,k)(x),(nh,k)(x))\le\sum_{i=1}^n d(((i-1)h,k)(x),(ih,k)(x))=nd((0,k)(x),(h,k)(x)),
        \end{equation*}
        therefore for each $h\in\Zds^b$, possibly taking a bigger $A_h$,
        \begin{equation*}
            \mleft((0,k)(x),(\Phi_\varepsilon(h),k)(x),\frac1\varepsilon\mright),((0,k)(x),(nh,k)(x),n)\in\Ccal_{A_h},\qquad \text{for any $\varepsilon\in(0,1]$}.
        \end{equation*}
        \zcref[S]{phiepslip} then yields the existence of a Lipschitz constant $\ell_{A_h}$ such that, if $\varepsilon\in(0,1]$,
        \begin{multline}\label{eq:effLex2}
            \mleft|\phi((0,k)(x),(nh,k)(x),n,\omega)-\phi\mleft((0,k)(x),(\Phi_\varepsilon(h),k)(x),\frac1\varepsilon,\omega\mright)\mright|\\
            \le\ell_{A_h}\mleft(d((\Phi_\varepsilon(h),k)(x),(nh,k)(x))+\mleft|n-\frac1\varepsilon\mright|\mright).
        \end{multline}
        Let $\{\varepsilon_m\}_{m\in\Nds}$ be a positive sequence converging to 0 and $\{n_m\}_{m\in\Nds}\subset\Nds$ be a such that $\dfrac1{n_m+1}<\varepsilon_m\le\dfrac1{n_m}$. Applying \zcref{normconebound} we get
        \begin{equation*}
            \varepsilon_m\mleft(d((\Phi_{\varepsilon_m}(h),k)(x),(n_m h,k)(x))+\mleft|n_m-\frac1{\varepsilon_m}\mright|\mright)\!\le\|\varepsilon_m\Phi_{\varepsilon_m}(h)-\varepsilon_m n_m h\|+|\varepsilon_m n_m-1|\xrightarrow{m\to\infty}0,
        \end{equation*}
        consequently~\zcref{eq:effLex1,eq:effLex2,eq:minactequiv} imply
        \begin{equation*}
            \lim_{\varepsilon\to0^+}\phi_\varepsilon((0,k)(x),(\Phi_\varepsilon(h),k)(x),1,\omega)=\ovL(h,k,x,\omega),\qquad \text{a.s.}.
        \end{equation*}
        Moreover, we have from \zcref{phiepslip} that the $\phi_\varepsilon(\omega)$ are locally $d_\varepsilon$-equiLipschitz continuous, thereby \zcref{eqcont} yields that $\ovL$ does not depend on $x$ and $k$. Another application of \zcref{eqcont} and \zcref{phistat}\ref{en:phistat2} further show that, for any $h\in\Zds^b$ and $g\in \G$,
        \begin{align*}
            \ovL(h,\tau_g\omega)&=\lim_{\varepsilon\to0^+}\phi_\varepsilon((0,k)(x),(\Phi_\varepsilon(h),k)(x),1,\tau_g\omega)\\
            &=\lim_{\varepsilon\to0^+}\phi_\varepsilon(((0,k)+g)(x),((\Phi_\varepsilon(h),k)+g)(x),1,\omega)=\ovL(h,\omega),
        \end{align*}
        hence the ergodic property of the dynamical system yields that $\ovL(h)$ is almost surely deterministic. This proves the claim when $h\in\Zds^b$.\\
        We now extend $\ovL$ to $\Qds^b$. More specifically, observe that if $h\in\Qds^b$ there is an $m\in\Nds$ so that $mh\in\Zds^b$ and that, reasoning as in the previous step, the limit
        \begin{equation*}
            \lim_{\varepsilon\to0^+}\frac1m\phi_{m\varepsilon}((0,k)(x),(\Phi_{m\varepsilon}(mh),k)(x),m,\omega)
        \end{equation*}
        exists, is almost surely deterministic and does not depend on the choice of $k$ and $x$. In addition, we get from \zcref{phiepsconvaux} that
        \begin{equation*}
            \ovL(h)\coloneqq\lim_{\varepsilon\to0^+}\phi_\varepsilon((0,k)(x),(\Phi_\varepsilon(h),k)(x),1,\omega)=\lim_{\varepsilon\to0^+}\frac1m\phi_{m\varepsilon}((0,k)(x),(\Phi_{m\varepsilon}(mh),k)(x),m,\omega)
        \end{equation*}
        is well-defined for any $h\in\Qds^b$.\\
        Finally, we notice that for each $h\in\Qds^b$ the set $\Omega_h$ made up of the $\omega\in\Omega$ such that~\eqref{eq:effLex.1} holds true is a set of full measure. Additionally,
        \begin{equation*}
            \ovOmega\coloneqq\bigcap_{h\in\Qds^b}\Omega_h
        \end{equation*}
        is a set of full measure too, thus the local $d_\varepsilon$-equiLipschitz continuity of $\phi_\varepsilon$ (see \zcref{phiepslip}), the density of $\Qds^b$ in $\Rds^b$ and~\eqref{eq:phiepsconv} show that the limit
        \begin{equation*}
            \ovL(h)\coloneqq\lim_{\varepsilon\to0^+}\phi_\varepsilon((0,k)(x),(\Phi_\varepsilon(h),k)(x),1,\omega)
        \end{equation*}
        exists for any $h\in\Rds^b$, $k\in\Tcal$, $x\in M$ and $\omega\in\ovOmega$, concluding the proof.
    \end{proof}

    \medskip

    The next result is a generalization of \zcref{effLex} which will play an important role later.

    \begin{cor}\label{phiunifconv}
        For any $\omega$ in a set of full measure, $h,h'\in\Rds^b$, $x,x'\in M$, $k,k'\in\Tcal$ and $t>0$
        \begin{equation*}
            \lim_{\varepsilon\to0^+}\phi_\varepsilon\mleft(\mleft(\Phi_\varepsilon\mleft(h'\mright),k'\mright)\mleft(x'\mright),(\Phi_\varepsilon(h),k)(x),t,\omega\mright)=t\ovL\mleft(\frac{h-h'}t\mright).
        \end{equation*}
    \end{cor}
    \begin{proof}
        The case $b=0$ is trivial, hence we assume that $b\ge1$.\\
        By \zcref[tlastsep={ and }]{effLex,phiepslip,normconebound,eq:phiepsconv}, for any $x\in M$, $k\in\Tcal$ $h,h'\in\Rds$ and a.e.\ $\omega\in\Omega$,
        \begin{align*}
            \lim_{\varepsilon\to0^+}\phi_\varepsilon\mleft((0,k)(x),\mleft(\Phi_\varepsilon(h)-\Phi_\varepsilon\mleft(h'\mright),k\mright)(x),1,\omega\mright)&=\lim_{\varepsilon\to0^+}\phi_\varepsilon\mleft((0,k)(x),\mleft(\Phi_\varepsilon\mleft(h-h'\mright),k\mright)(x),1,\omega\mright)\\
            &=\ovL\mleft(h-h'\mright).
        \end{align*}
        We fix $x\in M$, $k\in\Tcal$, $h,h'\in\Rds^b$ and apply Egoroff's Theorem to the previous identity: for every $\delta>0$ there is set $E_\delta\subseteq\Omega$ with $\Pds(\Omega\setminus E_\delta)<\delta$ such that
        \begin{equation}\label{eq:phiunifconv1}
            \lim_{\varepsilon\to0^+}\sup_{\omega\in E_\delta}\mleft|\phi_\varepsilon\mleft((0,k)(x),\mleft(\Phi_\varepsilon(h)-\Phi_\varepsilon\mleft(h'\mright),k\mright)(x),1,\omega\mright)-\ovL\mleft(h-h'\mright)\mright|=0.
        \end{equation}
        It follows from \zcref{lataux} that, fixed an $\omega_0$ in a set of full measure $\Omega_{h,h'}$ and a constant $\epsilon>0$, there exist two positive constants $\delta_\epsilon$ and $R_\epsilon$ such that, for any $\ovh\in\Rds^b$ with $\epsilon\mleft\|\ovh\mright\|>R_\epsilon$,
        \begin{equation}\label{eq:phiunifconv2}
            \mleft\|\ovh-h_0\mright\|<\epsilon\mleft\|\ovh\mright\|,\qquad \text{for some $h_0$ with }\tau_{h_0}\omega_0\in E_{\delta_\epsilon}.
        \end{equation}
        There is then an $\varepsilon_\epsilon>0$ small enough such that, for any $\varepsilon\in(0,\varepsilon_\epsilon)$, $\epsilon\mleft\|\dfrac{h'}\varepsilon\mright\|>R_\epsilon$ and, by~\eqref{eq:phiunifconv1},
        \begin{equation*}
            \sup_{\omega\in E_{\delta_\epsilon}}\mleft|\phi_\varepsilon\mleft((0,k)(x),\mleft(\Phi_\varepsilon(h)-\Phi_\varepsilon\mleft(h'\mright),k\mright)(x),1,\omega\mright)-\ovL\mleft(h-h'\mright)\mright|<\epsilon,
        \end{equation*}
        thus exploiting~\eqref{eq:phiunifconv2} we set $h_\varepsilon\in\Zds^b$ so that, for any $\varepsilon\in(0,\varepsilon_\epsilon)$, $\tau_{h_\varepsilon}\omega_0\in E_{\delta_\epsilon}$ and
        \begin{equation}\label{eq:phiunifconv3}
            \mleft\|h'-\varepsilon h_\varepsilon\mright\|\le\epsilon\mleft\|h'\mright\|.
        \end{equation}
        This and \zcref{phistat}\ref{en:phistat2} yield that, for any $\varepsilon\in(0,\varepsilon_\epsilon)$,
        \begin{multline}\label{eq:phiunifconv4}
            \mleft|\phi_\varepsilon\mleft((h_\varepsilon,k)(x),\mleft(\Phi_\varepsilon(h)-\Phi_\varepsilon\mleft(h'\mright)+h_\varepsilon,k\mright)(x),1,\omega_0\mright)-\ovL\mleft(h-h'\mright)\mright|\\
            =\mleft|\phi_\varepsilon\mleft((0,k)(x),\mleft(\Phi_\varepsilon(h)-\Phi_\varepsilon\mleft(h'\mright),k\mright)(x),1,\tau_{h_\varepsilon}\omega_0\mright)-\ovL\mleft(h-h'\mright)\mright|<\epsilon.
        \end{multline}
        We have from \zcref{eq:phiunifconv3} that $\varepsilon h_\varepsilon\to h'$ as $\epsilon,\varepsilon\to0^+$, thereby \zcref[tlastsep={ and }]{phiepslip,eqcont,eq:phiunifconv4} show that, for any $k'\in\Tcal$ and $x'\in M$,
        \begin{multline}\label{eq:phiunifconv5}
            \lim_{\varepsilon\to0^+}\phi_\varepsilon\mleft(\mleft(\Phi_\varepsilon\mleft(h'\mright),k'\mright)\mleft(x'\mright),(\Phi_\varepsilon(h),k)(x),1,\omega_0\mright)\\
            =\lim_{\varepsilon\to0^+}\phi_\varepsilon\mleft((h_\varepsilon,k)(x),\mleft(\Phi_\varepsilon(h)-\Phi_\varepsilon\mleft(h'\mright)+h_\varepsilon,k\mright)(x),1,\omega_0\mright)=\ovL\mleft(h-h'\mright).
        \end{multline}
        We stress out that this identity holds true for every $\omega_0\in\Omega_{h,h'}$, which is a set of full measure depending on $h,h'\in\Rds^b$. Now define the set of full measure
        \begin{equation*}
            \ovOmega\coloneqq\bigcap_{h,h'\in\Qds^b}\Omega_{h,h'},
        \end{equation*}
        and observe that the density of $\Qds^b$ in $\Rds^b$, \zcref{phiepslip} and the continuity of $\ovL$ (follows from \zcref{equicontconvcont}) imply that~\eqref{eq:phiunifconv5} holds for any $k,k'\in\Tcal$, $x,x'\in M$, $h,h'\in\Rds^b$ and $\omega_0\in\ovOmega$. We finally apply \zcref{phiepsconvaux} to get that, for any $k,k'\in\Tcal$, $x,x'\in M$, $h,h'\in\Rds^b$, $t>0$ and $\omega\in\ovOmega$,
        \begin{multline*}
            \lim_{\varepsilon\to0^+}\phi_\varepsilon\mleft(\mleft(\Phi_\varepsilon\mleft(h'\mright),k'\mright)\mleft(x'\mright),(\Phi_\varepsilon(h),k)(x),t,\omega\mright)\\
            =t\lim_{\varepsilon\to0^+}\phi_{\frac\varepsilon t}\mleft(\mleft(\Phi_{\frac\varepsilon t}\mleft(\frac{h'}t\mright),k'\mright)\mleft(x'\mright),\mleft(\Phi_{\frac\varepsilon t}\mleft(\frac ht\mright),k\mright)(x),1,\omega\mright)=t\ovL\mleft(\frac{h-h'}t\mright),
        \end{multline*}
        which concludes the proof.
    \end{proof}

    \medskip

    Next we provide a key property of the effective Lagrangian.

    \begin{prop}\label{effLprop}
        The function $\ovL:\Rds^b\to\Rds$ is convex and there exist two convex nondecreasing superlinearly coercive functions $\vartheta_{\ovL},\Theta_{\ovL}:\Rds^+\to\Rds^+$ such that
        \begin{equation}\label{eq:effLprop.1}
            \vartheta_{\ovL}(\|h\|)\le\ovL(h)\le\Theta_{\ovL}(\|h\|),\qquad \text{for any }h\in\Rds^b.
        \end{equation}
    \end{prop}
    \begin{proof}
        By \zcref{phiunifconv} we have that, for any $x\in M$, $h,h'\in\Rds^b$, $\lambda\in[0,1]$ and a.e.\ $\omega\in\Omega$,
        \begin{align*}
            \ovL\mleft(\lambda h+(1-\lambda)h'\mright)&=\lim_{\varepsilon\to0^+}\phi_\varepsilon\mleft((\Phi_\varepsilon(-\lambda h),0)(x),\mleft(\Phi_\varepsilon\mleft((1-\lambda)h'\mright),0\mright)(x),1,\omega\mright)\\
            &\le\lim_{\varepsilon\to0^+}\phi_\varepsilon\mleft((\Phi_\varepsilon(-\lambda h),0)(x),x,\lambda,\omega\mright)+\phi_\varepsilon\mleft(x,\mleft(\Phi_\varepsilon\mleft((1-\lambda)h'\mright),0\mright)(x),1-\lambda,\omega\mright)\\
            &=\lambda\ovL(h)+(1-\lambda)\ovL\mleft(h'\mright),
        \end{align*}
        {\it i.e.}, $\ovL$ is convex. \zcref[S]{eq:effLprop.1} is a straightforward consequence of \zcref{eq:phiepsconv,effLex,phibound,normconebound}.
    \end{proof}

    \medskip

    We conclude this section with the following property of the effective Lagrangian, namely it inherits the homogeneity of the Lagrangian/Hamiltonian in the velocity/momentum variable. This --- in addition to its own interest --- will be used in \zcref{sec:stablenorm} for discussing a notion of stable norm for stationary ergodic families of metrics.

    \begin{prop}\label{homogeneityL}
        If the Lagrangian $L:TM\times\Omega\to\Rds$ defined in~\eqref{eq:lagequiv} satisfies
        \begin{equation*}
            L(x,\lambda q,\omega)=|\lambda|^\alpha L(x,q,\omega),\qquad \text{for any $\lambda\in\Rds$},
        \end{equation*}
        for some $\alpha>0$, then
        \begin{equation}\label{eq:homogeneityL.1}
            \ovL(\lambda h)=|\lambda|^\alpha \ovL(h),\qquad \text{for every $h\in\Rds^b$ and $\lambda\in\Rds$}.
        \end{equation}
    \end{prop}

    \medskip

    \begin{proof}
        First we fix $h\in\Rds^b$ and $\lambda\ne0$. Given a curve $\gamma$ so that
        \begin{equation*}
            \varepsilon\int_0^{\frac{1}{\varepsilon}}L(\gamma,\dot\gamma,\omega)\drm\tau=\varepsilon\phi\mleft((0,x),(\Phi_\varepsilon(\lambda h),x),\frac{1}{\varepsilon},\omega\mright)=\phi_\varepsilon((0,x),(\Phi_\varepsilon(\lambda h),x),1,\omega),
        \end{equation*}
        we have by \zcref{phiunifconv} that
        \begin{equation}\label{eq:homogeneityL1}
            \ovL(\lambda h)=\lim_{\varepsilon\to0^+}\phi_\varepsilon((0,x),(\Phi_\varepsilon(\lambda h),x),1,\omega)=\lim_{\varepsilon\to0^+}\varepsilon\int_0^{\frac{1}{\varepsilon}}L(\gamma,\dot\gamma,\omega)\drm\tau.
        \end{equation}
        We make the change of variable $r=|\lambda|\tau$ to get
        \begin{align*}
            \varepsilon\int_0^{\frac{1}{\varepsilon}}L(\gamma,\dot\gamma,\omega)\drm\tau&=\varepsilon\int_0^{\frac{|\lambda|}{\varepsilon}}L\mleft(\gamma\mleft(\frac{r}{|\lambda|}\mright),\dot\gamma\mleft(\frac{r}{|\lambda|}\mright),\omega\mright)\frac{\drm r}{|\lambda|}\\
            &=|\lambda|^\alpha \frac{1}{|\lambda|}\varepsilon\int_0^{\frac{|\lambda|}{\varepsilon}}L\mleft(\gamma\mleft(\frac{r}{|\lambda|}\mright),\dot\gamma\mleft(\frac{r}{|\lambda|}\mright)\frac{1}{|\lambda|},\omega\mright)\drm r\\
            &\ge|\lambda|^\alpha \frac{1}{|\lambda|}\phi_\varepsilon((0,x),(\Phi_\varepsilon(\lambda h),x),|\lambda|,\omega),
        \end{align*}
        which in turn shows, see \zcref{eq:homogeneityL1,phiunifconv}, that
        \begin{equation*}
            \ovL(\lambda h)\ge|\lambda|^\alpha\lim_{\varepsilon\to0^+}\frac{1}{|\lambda|}\phi_\varepsilon((0,x),(\Phi_\varepsilon(\lambda h),x),|\lambda|,\omega)=|\lambda|^\alpha \ovL(h).
        \end{equation*}
        Setting $k=\lambda h$ and $\mu=\lambda^{-1}$ we further obtain
        \begin{equation*}
            |\mu|^\alpha \ovL(k)\ge\ovL(\mu k).
        \end{equation*}
        Since $h$ and $\lambda$ are arbitrary, these inequalities prove~\zcref{eq:homogeneityL.1} when $\lambda\ne0$. The case $\lambda=0$ follows from the continuity of $\ovL$.
    \end{proof}

    \medskip

    \section{Proof of the Main Theorem}\label{sec:5}

    We can now define a Hamilton--Jacobi type problem whose solutions are limits of the solutions to~\eqref{eq:HJeps}.

    First we define the \emph{effective Hamiltonian} $\ovH$ as the convex conjugate of $\ovL$. It follows from \zcref{effLprop} and standard properties of the convex conjugates that also $\ovH$ is convex and superlinearly coercive. In the deterministic case $\ovH$ corresponds to Mather's $\alpha$ function (see, for instance, \cite{Sorrentino15}).

    We have the following PDE problem on the asymptotic cone $\Rds^b$: find a function $v:\Rds^b\times\Rds^+\to\Rds$ such that
    \begin{equation}\label{eq:HJhom}\tag{\=HJ}
        \mleft\{
        \begin{aligned}
            &\partial_t v(h,t)+\ovH(\partial_h v)=0,&& \text{in $\Rds^b\times(0,\infty)$},\\
            &v(h,0)=v_0(h),&& \text{in $\Rds^b$},
        \end{aligned}
        \mright.
    \end{equation}
    where $v_0$ is a uniformly continuous function.

    The next result is well known, see for instance~\cite{Lions82}.

    \begin{theo}
        Under our assumptions, \zcref{eq:HJhom} admits a unique uniformly continuous solution, which is given by the Lax--Oleinik type formula
        \begin{equation}\label{eq:vfhom}
            v(h,t)=\inf_{h'\in\Rds^b}\mleft\{v_0\mleft(h'\mright)+t\ovL\mleft(\frac{h-h'}t\mright)\mright\},\qquad \text{for $(h,t)\in\Rds^b\times\Rds^+$}.
        \end{equation}
    \end{theo}

    \bigskip

    For the reader's sake, we recall here the statement of the Main Theorem (in which we incorporate the notions of convergence and equicontinuity introduced in Section~\ref{sec:3.2}):

    \begin{main}
        \noindent Given a Hamiltonian $H:T^*M\times\Omega\to\Rds$ satisfying \zcref[comp]{condcoerc,condcont,condconv,condreg,condstat}, there exists a convex and superlinear effective Hamiltonian $\ovH:\Rds^b\to\Rds$ such that the following holds.

        \smallskip

        \noindent Let $\{u_0^\varepsilon\}_{\varepsilon>0}$ be a collection of uniformly $d_\varepsilon$-equicontinuous random functions from $M_\varepsilon$ into $\Rds$. If $u_0^\varepsilon(\omega)$ $\varepsilon\Phi_\varepsilon$-converges pointwise to $v_0$ a.s., then, for any $\omega$ in a set of full measure $\ovOmega\subseteq\Omega$, the solutions $u_\varepsilon(\omega)$ to~\eqref{eq:HJeps} locally uniformly $\varepsilon\Phi_\varepsilon$-converge a.s.\ to the only uniformly continuous solution $v$ to~\eqref{eq:HJhom}.
    \end{main}

    \medskip

    To prove this theorem we need some auxiliary results.

    \begin{lem}\label{minvfeps}
        Let $u_0^\varepsilon:M_\varepsilon\times\Omega\to\Rds$ be a collection of uniformly $d_\varepsilon$-equicontinuous random functions indexed by $\varepsilon>0$. For any $\varepsilon>0$, $\omega\in\Omega$ and $(x,t)\in M_\varepsilon\times\Rds^+$, the value function~\eqref{eq:vfeps} admits a minimizer $x_\varepsilon$ and there is a concave modulus of continuity $\rho$ so that
        \begin{equation*}
            d_\varepsilon(x_\varepsilon,x)\le\rho(t),
        \end{equation*}
        for every $\varepsilon>0$, $\omega\in\Omega$, $(x,t)\in M_\varepsilon\times\Rds^+$ and optimal $x_\varepsilon$ for the value function~\eqref{eq:vfeps} at $(x,t,\omega)$.
    \end{lem}
    \begin{proof}
        By definition each $u_0^\varepsilon$ admits a common concave modulus of continuity $\sigma$, hence for each $\delta>0$ there is an $A_\delta>0$ so that
        \begin{equation*}
            \sigma(r)\le A_\delta r+\delta,\qquad \text{for all $r\in\Rds^+$}.
        \end{equation*}
        Then, we let $\vartheta_L$ and $\Theta_L$ be as in \zcref{phibound}. Thanks to the coercivity of $\vartheta_L$ we also have a constant $B_\delta\ge0$ depending only on $\delta$ and $\vartheta_L$ satisfying
        \begin{equation*}
            \vartheta_L(r)\ge\mleft(A_\delta+\sqrt\delta\mright)r-B_\delta,\qquad \text{for any $\delta>0$ and $r\in\Rds^+$}.
        \end{equation*}
        Thus, for any $\varepsilon>0$, $\omega\in\Omega$ and $x,x'\in M_\varepsilon$,
        \begin{equation}\label{eq:minvfeps1}
            u^\varepsilon_0\mleft(x',\omega\mright)-u_0^\varepsilon(x,\omega)\ge-\sigma\mleft(d_\varepsilon\mleft(x,x'\mright)\mright)\ge-A_\delta d_\varepsilon\mleft(x,x'\mright)-\delta,
        \end{equation}
        while \zcref{phibound} yields, for any $\varepsilon>0$, $\omega\in\Omega$, $x,x'\in M_\varepsilon$ and $t>0$,
        \begin{equation}\label{eq:minvfeps2}
            \mleft(A_\delta+\sqrt\delta\mright)d_\varepsilon\mleft(x',x\mright)-tB_\delta\le t\vartheta_L\mleft(\frac{d_\varepsilon(x',x)}t\mright)\le\phi_\varepsilon\mleft(x',x,t,\omega\mright).
        \end{equation}
        Next we fix $\varepsilon>0$, $x\in M_\varepsilon$, $t\in\Rds^+$, $\omega\in\Omega$ and let $x'\in M_\varepsilon$ be such that
        \begin{equation}\label{eq:minvfeps3}
            u_0^\varepsilon\mleft(x',\omega\mright)+\phi_\varepsilon\mleft(x',x,t,\omega\mright)\le u_0^\varepsilon(x,\omega)+\phi_\varepsilon(x,x,t,\omega)\le u_0^\varepsilon(x,\omega)+t\Theta_L(0).
        \end{equation}
        Finally~\zcref{eq:minvfeps1,eq:minvfeps2,eq:minvfeps3} show that
        \begin{equation*}
            \sqrt\delta d_\varepsilon\mleft(x',x\mright)\le u_0^\varepsilon\mleft(x',\omega\mright)-u_0^\varepsilon(x,\omega)+\delta+\phi_\varepsilon\mleft(x',x,t,\omega\mright)+tB_\delta\le t\Theta_L(0)+tB_\delta+\delta,
        \end{equation*}
        which implies that there is a concave modulus of continuity $\rho$ such that
        \begin{equation*}
            d_\varepsilon\mleft(x',x\mright)\le\rho(t)\coloneqq\inf_{\delta>0}\mleft\{\mleft(\frac{\Theta_L(0)+B_\delta}{\sqrt\delta}\mright)t+\sqrt\delta\mright\},
        \end{equation*}
        for every $\varepsilon>0$, $\omega\in\Omega$, $x\in M_\varepsilon$ and $t\in\Rds^+$ whenever $x'\in M_\varepsilon$ satisfies~\eqref{eq:minvfeps3}. This proves the claim since the infimum in~\eqref{eq:vfeps} can be restricted to the compact subset of $M_\varepsilon$ made up by the $x'\in M_\varepsilon$ satisfying~\eqref{eq:minvfeps3}.
    \end{proof}

    \begin{rem}\label{minvfhom}
        Arguing as in the proof of \zcref{minvfeps}, with \zcref{effLprop} in place of \zcref{phibound}, we can prove that the infimum in~\eqref{eq:vfhom} is actually a minimum.
    \end{rem}

    \begin{prop}\label{vfloclip}
        Given a collection $u_0^\varepsilon:M_\varepsilon\times\Omega\to\Rds$ of uniformly $d_\varepsilon$-equicontinuous random functions indexed by $\varepsilon>0$, for each $t_0>0$ the value functions $u_\varepsilon(\omega)$ defined in~\eqref{eq:vfeps} and restricted to $M_\varepsilon\times[t_0,\infty)$ are $d_\varepsilon$-equi--Lipschitz continuous, uniformly with respect to $\omega$.
    \end{prop}
    \begin{proof}
        Fixed a $t_0>0$, \zcref{minvfeps} and the concavity of $\rho$ yield that, for any $\varepsilon>0$, $\omega\in\Omega$, $(x,t)\in M_\varepsilon\times[t_0,\infty)$ and $x_\varepsilon$ optimal for $u_\varepsilon(x,t,\omega)$, there is a positive constant $A$ so that
        \begin{equation*}
            d_\varepsilon\mleft(x_\varepsilon,x\mright)\le\rho(t)\le At+1\le\mleft(A+\frac1{t_0}\mright)t.
        \end{equation*}
        It follows that we can select a $C>A$ such that, for any $\varepsilon>0$, $\omega\in\Omega$ and $(x,t)\in M_\varepsilon\times[t_0,\infty)$,
        \begin{equation}\label{eq:vfloclip1}
            u_\varepsilon(x,t,\omega)=\inf\mleft\{u_0^\varepsilon\mleft(x',\omega\mright)+\phi_\varepsilon\mleft(x',x,t,\omega\mright):x'\in X,\,d_\varepsilon\mleft(x',x\mright)<Ct\mright\}.
        \end{equation}
        Taking into account \zcref{phiepslip}, \zcref{eq:vfloclip1} shows that $u_\varepsilon(x,t,\omega)$ is the infimum of $d_\varepsilon$-equiLipschitz continuous functions in a neighborhood of $(x,t)$ uniformly with respect to $\omega\in\Omega$. Moreover, the Lipschitz constant is independent of $(x,t)$. This concludes the proof.
    \end{proof}

    \begin{prop}\label{vfepsequicont}
        If the initial data $u_0^\varepsilon:M_\varepsilon\times\Omega\to\Rds$ are uniformly $d_\varepsilon$-equicontinuous random functions, then the corresponding value functions $u_\varepsilon$, defined in~\eqref{eq:vfeps}, are uniformly $d_\varepsilon$-equicontinuous random functions.
    \end{prop}
    \begin{proof}
        We assume by contradiction that fixed $\delta>0$ there exist the subsequences $\{\varepsilon_n\}_{n\in\Nds}\subset(0,\infty)$, $\{(x_n,t_n)\}_{n\in\Nds},\{(x_n',t_n')\}_{n\in\Nds}\subset M\times\Rds^+$ and $\{\omega_n\}_{n\in\Nds}\subset\Omega$ so that
        \begin{equation}\label{eq:vfepsequicont1}
            \lim_{n\to\infty}d_{\varepsilon_n}\mleft(x_n,x_n'\mright)+\mleft|t_n-t_n'\mright|=0
        \end{equation}
        and
        \begin{equation}\label{eq:vfepsequicont2}
            \mleft|u_{\varepsilon_n}(x_n,t_n,\omega_n)-u_{\varepsilon_n}\mleft(x_n',t_n',\omega_n\mright)\mright|>\delta,\qquad \text{for any $n\in\Nds$}.
        \end{equation}
        We infer from \zcref{vfloclip} that, up to subsequences, both $t_n$ and $t_n'$ are infinitesimal; we thus proceed estimating
        \begin{equation*}
            |u_\varepsilon(x,t,\omega)-u_0^\varepsilon(x,\omega)|,\qquad \text{for $\varepsilon>0$, $x\in M_\varepsilon$, $t\in\Rds^+$ and $\omega\in\Omega$}.
        \end{equation*}
        Let $\sigma$ be the common modulus of continuity of the $u_0^\varepsilon(\omega)$ and $x'$ be an optimal point for $\varepsilon$, $(x,t)$ and $\omega$ in~\eqref{eq:vfeps}, then \zcref{phibound,minvfeps} yield
        \begin{equation}\label{eq:vfepsequicont3}
            \begin{aligned}
                u_0^\varepsilon(x,\omega)-u_\varepsilon(x,t,\omega)&\le u_0^\varepsilon(x,\omega)-u_0^\varepsilon\mleft(x',\omega\mright)-\phi_\varepsilon\mleft(x',x,t,\omega\mright)\le\sigma\mleft(d_\varepsilon\mleft(x',x\mright)\mright)-t\min_{r\in\Rds^+}\vartheta_L(r)\\
                &\le\sigma(\rho(t))+t\mleft|\min_{r\in\Rds^+}\vartheta_L(r)\mright|,
            \end{aligned}
        \end{equation}
        where $\rho$ is a modulus of continuity and $\vartheta_L$ is a convex and coercive function. We further get
        \begin{equation}\label{eq:vfepsequicont4}
            u_\varepsilon(x,t,\omega)-u_0^\varepsilon(x,\omega)\le u_0^\varepsilon(x,\omega)+\phi_\varepsilon(x,x,t,\omega)-u_0^\varepsilon(x,\omega)\le t\Theta_L(0).
        \end{equation}
        Finally, we apply \zcref{eq:vfepsequicont3,eq:vfepsequicont4} to~\eqref{eq:vfepsequicont2} and obtain
        \begin{multline*}
            \mleft|u_{\varepsilon_n}(x_n,t_n,\omega_n)-u_{\varepsilon_n}\mleft(x_n',t_n',\omega_n\mright)\mright|\\
            \begin{aligned}
                \le\,&|u_{\varepsilon_n}(x_n,t_n,\omega_n)-u_0^{\varepsilon_n}(x_n,\omega_n)|+\mleft|u_0^{\varepsilon_n}(x_n,\omega_n)-u_0^{\varepsilon_n}\mleft(x_n',\omega_n\mright)\mright|+\mleft|u_0^{\varepsilon_n}\mleft(x_n',\omega_n\mright)-u_{\varepsilon_n}\mleft(x_n',t_n',\omega_n\mright)\mright|\\
                \le\,&\sigma(\rho(t_n))+t_n\mleft|\min_{r\in\Rds^+}\vartheta_L(r)\mright|+t_n\Theta_L(0)+\sigma\mleft(d_\varepsilon\mleft(x_n',x_n\mright)\mright)+\sigma\mleft(\rho\mleft(t'_n\mright)\mright)+t_n'\mleft|\min_{r\in\Rds^+}\vartheta_L(r)\mright|+t_n'\Theta_L(0).
            \end{aligned}
        \end{multline*}
        Since both $t_n$ and $t_n'$ are infinitesimal, this and~\eqref{eq:vfepsequicont1} contradict~\eqref{eq:vfepsequicont2}.
    \end{proof}

    We can now prove the Main Theorem.

    \begin{proof}[Proof of the Main Theorem]
        We denote with $\Omega_0$ the set of full measure such that
        \begin{equation}\label{eq:convsol1}
            \lim_{\varepsilon\to0^+}u_0^\varepsilon((\Phi_\varepsilon(h),k)(x),\omega)=v_0(h),\qquad \text{for any $x\in M$, $h\in\Rds^b$, $k\in\Tcal$ and $\omega\in\Omega_0$}.
        \end{equation}
        By \zcref{phiunifconv} we also have that, for every $\omega$ in a set of full measure $\Omega_1$, $x,x'\in M$, $h,h'\in\Rds^b$, $k,k'\in\Tcal$ and $t>0$,
        \begin{equation}\label{eq:convsol2}
            \lim_{\varepsilon\to0^+}\phi_\varepsilon\mleft(\mleft(\Phi_\varepsilon\mleft(h'\mright),k'\mright)\mleft(x'\mright),(\Phi_\varepsilon(h),k)(x),t,\omega\mright)=t\ovL\mleft(\dfrac{h-h'}t\mright).
        \end{equation}
        $\ovOmega\coloneqq\Omega_0\cap\Omega_1$ is a set of full measure, thus if we can prove the pointwise $\varepsilon\Phi_\varepsilon$-convergence of the maps $u_\varepsilon(\omega)$ to $v$ for any $\omega\in\ovOmega$, our claim will follow from \zcref{vfepsequicont,equicontconvlocunif}.\\
        Since $u_\varepsilon$ and $v$ are, respectively, solutions to~\eqref{eq:HJeps} and~\eqref{eq:HJhom}, we have that $u_\varepsilon(\cdot,0,\omega)=u_0^\varepsilon(\cdot,\omega)$ and $v(\cdot,0)=v_0(\cdot)$, {\it i.e.}, the convergence for $t=0$ is just a consequence of our assumptions. We hence fix $t>0$, $x\in M$ and $h\in\Rds^b$. Given a $h'\in\Rds^b$ so that (see \zcref{minvfhom})
        \begin{equation*}
            v(h,t)=v_0\mleft(h'\mright)+t\ovL\mleft(\frac{h-h'}t\mright),
        \end{equation*}
        we have by~\eqref{eq:convsol1} and~\eqref{eq:convsol2} that, for any $\omega\in\ovOmega$,
        \begin{equation}\label{eq:convsol3}
            \begin{aligned}
                v(h,t)=\,&\lim_{\varepsilon\to0^+}u_0^\varepsilon\mleft(\mleft(\Phi_\varepsilon\mleft(h'\mright),0\mright)(x),\omega\mright)+\phi_\varepsilon\mleft(\mleft(\Phi_\varepsilon\mleft(h'\mright),0\mright)(x),(\Phi_\varepsilon(h),k)(x),t,\omega\mright)\\
                \ge\,&\limsup_{\varepsilon\to0^+}u_\varepsilon((\Phi_\varepsilon(h),0)(x),t,\omega).
            \end{aligned}
        \end{equation}
        Now consider an optimal point for $u_\varepsilon((\Phi_\varepsilon(h),0)(x),t,\omega)$ in~\eqref{eq:vfeps}. Since the action $\G$ is co-compact, we can assume that it takes the form $(h_\varepsilon,k_\varepsilon)(x_\varepsilon)$ where $h_\varepsilon\in\Zds^b$, $k_\varepsilon\in\Tcal$ and $x_\varepsilon\in M_\varepsilon$ is such that $d_\varepsilon(x_\varepsilon,x)\to0$ as $\varepsilon\to0^+$. We have from~\zcref{eq:phiepsconv,minvfeps,normconebound} that there exist a constant $C_1$, an infinitesimal sequence $\{C_\varepsilon\}_{\varepsilon>0}$ and a modulus of continuity $\rho$ such that
        \begin{equation*}
            \|h-\varepsilon h_\varepsilon\|\le C_1d_\varepsilon((h_\varepsilon,k_\varepsilon)(x_\varepsilon),(\Phi_\varepsilon(h),0)(x))+C_\varepsilon\le C_1\rho(t)+C_\varepsilon,
        \end{equation*}
        which in turn implies that $\varepsilon h_\varepsilon$ converges, up to subsequences, to some $h''\in\Rds^b$. We then deduce from~\zcref[pairsep={, }]{eq:convsol1,eq:convsol2,eqcont} that, for any $\omega\in\ovOmega$,
        \begin{equation}\label{eq:convsol4}
            \begin{aligned}
                \liminf_{\varepsilon\to0^+}u_\varepsilon((\Phi_\varepsilon(h),0)(x),t,\omega)&=\lim_{\varepsilon\to0^+}\!u_0^\varepsilon((h_\varepsilon,k_\varepsilon)(x_\varepsilon),\omega)+\phi_\varepsilon((h_\varepsilon,k_\varepsilon)(x),(\Phi_\varepsilon(h),0)(x),t,\omega)\\
                &=v_0\mleft(h''\mright)+t\ovL\mleft(\frac{h-h''}t\mright)\ge v(h,t).
            \end{aligned}
        \end{equation}
        This concludes the proof since~\zcref[pairsep={, }]{eq:convsol3,eq:convsol4,eqcont} show that $u_\varepsilon(\omega)$ pointwise $\varepsilon\Phi_\varepsilon$-converges to $v$ for every $\omega\in\ovOmega$.
    \end{proof}

    \medskip

    \begin{rem}
        When $M$ is compact its asymptotic cone is a point. Under the same assumptions of the Main Theorem we further have that both the effective Lagrangian $\ovL$ and the effective Hamiltonian $\ovH$ are both constant. Then $u_\varepsilon$ will uniformly $\varepsilon\Phi_\varepsilon$-converge a.s.\ to the unique solution of the ODE
        \begin{equation*}
            \mleft\{
            \begin{aligned}
                &\partial_t v(t)=c,&& \text{in $(0,\infty)$},\\
                &v(0)=v_0,
            \end{aligned}
            \mright.
        \end{equation*}
        where $c$ and $v_0$ are constants.
    \end{rem}

    \appendix

    \section{Examples}\label{sec:exe}

    Let us describe some motivating examples.\\

    \noindent{\bf 1. (Periodic case)}
    Let $H:T^*M\times\Omega\to\Rds$ be a Hamiltonian satisfying \zcref[comp]{condcont,condconv,condreg,condcoerc,condstat} and assume that $(\Omega,\Sigma,\Pds)$ is a trivial probability space, namely $\Omega=\{\omega_0\}$, $\Sigma=\{\emptyset,\Omega\}$ and $\Pds=\delta_{\{\omega_0\}}$. Considering its projection $\whH$ on $M/\G$, see~\zcref{eq:Hproj}, we have by~\zcref{eq:lift2rlztn} that
    \begin{equation*}
        H(x,p,\omega_0)=\whH(\pi^*(x,p),\omega_0),\qquad \text{for any $(x,p)\in T^*M$},
    \end{equation*}
    where $\pi^*$ is the inverse of the pullback by the projection $\pi:M\to M/\G$. In view of \zcref{liftisdet}, this proves that every a.s.\ deterministic periodic Hamiltonian on a non-compact connected smooth finite-dimensional manifold without boundary can be obtained from the realizations of an {a.s.\ deterministic} Hamiltonian on a compact manifold.

    \bigskip

    \bigskip

    \noindent{\bf 2. (Quasi-periodic foliation on a torus)} Let us consider non-periodic examples. Let us start from a quasi-periodic translation on a torus and the corresponding induced foliation (this will be generalized afterwards).

    Let $\Tds\coloneqq\Rds/ \Zds$ and consider a Hamiltonian ${H}: \Tds^{n+1} \times \Rds^{n+1} \longrightarrow \Rds$. Let $(\alpha,1) \in \Rds^{n+1}$ be a non-resonant vector, namely $(\alpha,1) \cdot \nu \ne 0$ for every $\nu \in \Zds^{n+1}\setminus \{0\}$. \\
    Let us consider $(\Omega,\Sigma, \Pds)$, where $\Omega = \Tds$, $\Pds$ the Lebesgue (probability) measure on $\Tds$ and $\Sigma$ the $\sigma$-algebra of the Lebesgue measurable subsets of $\Tds$.\\
    We consider the following maps:
    \begin{eqnarray*}
        \tau: \Zds^n \times \Omega &\longrightarrow& \Omega \\
        (h, \omega) &\longmapsto& \tau_h(\omega)\coloneqq\omega + \alpha\cdot h \qquad ({\rm mod.}\, \Zds).
    \end{eqnarray*}
    For every $h \in \Zds^n$ the map $\tau_h: \Omega \longrightarrow \Omega$ is clearly measurable; moreover one can check that it preserves $\Pds$ and it is ergodic (it follows from the non-resonance condition on $(\alpha,1)$ and Kronecker's Theorem).

    We define:
    \begin{eqnarray*}
        H_\alpha: \Rds^{n} \times \Rds^n \times \Omega &\longrightarrow& \Rds\\
        (x, p, \omega) &\longmapsto& {H}(x, \, \omega+ x\cdot \alpha,\, p,\, p\cdot \alpha).
    \end{eqnarray*}

    \medskip

    Let us also consider the following action of $\Zds^n$ on $\Rds^{n} \times \Rds^n$
    \begin{eqnarray*}
        \Zds^n \times (\Rds^{n} \times \Rds^n) &\longrightarrow& (\Rds^{n} \times \Rds^n)\\
        (h, (x,p)) &\longmapsto& h^*(x,p)\coloneqq(x+h, p).
    \end{eqnarray*}

    \medskip

    For $h\in \Zds^n$, we have (we use the fact that $H$ is $\Zds^{n+1}$-periodic):
    \begin{eqnarray*}
        H_\alpha(h^* (x,p),\, \omega) &=& H(x + h, \, \omega+ (x+h)\cdot \alpha,\, p,\, p\cdot \alpha) \\
        &=& H(x, \, (\omega+ h\cdot \alpha) + x\cdot \alpha,\, p,\, p\cdot \alpha)\\
        &=& H(x, \, \tau_h(\omega) + x\cdot \alpha,\, p,\, p\cdot \alpha)\\
        &=& H_\alpha(x,\,p,\, \tau_h(\omega)).
    \end{eqnarray*}

    \medskip

    \begin{rem}
        Geometrically the above example can be interpreted in the following way. We consider a foliation of $\Tds^{n+1}$ with leaves
        \begin{equation*}
            \mathcal F(\omega) \coloneqq\mleft\{(x_1,\dotsc, x_{n+1})\in \Tds^{n+1}: x_{n+1} = \omega + \alpha \cdot (x_1,\dotsc,x_n)\;(\mathrm{mod}.\;\Zds)\mright\}, \qquad \text{for $\omega\in \Tds$}.
        \end{equation*}
        In particular
        \begin{equation*}
            \mleft\{
            \begin{array}{lll}
                \mathcal F(\omega) \equiv \mathcal F(\omega') & & \text{if $\exists\; h\in \Zds^n:\; \tau_h(\omega)=\omega'$}\\
                \mathcal F(\omega) \cap \mathcal F(\omega') = \emptyset && \text{otherwise}.
            \end{array}
            \mright.
        \end{equation*}
        Each of these leaves is dense on $\Tds^{n+1}$. $\Omega$ can be interpreted as a transversal section to this foliation, for example identified with $\{x_1 = \dotsb = x_n= 0\}$, and each Hamiltonian $H_\alpha (\cdot, \cdot, \omega)$ corresponds to the Hamiltonian $H$ restricted to $T^*\mathcal F(\omega)$; the dynamical system induced on $\Omega$ is related to the intersection of the leaves with the transversal section.
    \end{rem}

    \bigskip

    \noindent{\bf 3. (Foliated manifolds)} The above example can be generalized, under suitable assumptions, to other examples of foliated manifolds.

    Let $\Tds\coloneqq\Rds/ \Zds$ and let $M$ be the maximal abelian covering of a closed manifold $\whM$, and denote by $\hat\pi: M\longrightarrow \whM$ the covering map. Consider a Hamiltonian ${H}: T^*(\whM \times \Tds ) \longrightarrow \Rds$. We denote by $(x,\theta) \in \whM\times \Tds$.\\
    Let $\alpha$ be a closed non-exact $1$-form on $\whM$ and, since the lift of $\alpha$ to $M$ is exact,
    let us consider one of its primitive $F_\alpha: M \longrightarrow \Rds$.
    This allows us to define a foliation of $\whM\times \Tds$ by leaves that are diffeomorphic to $M$; namely, for every $\omega\in \Tds$ we consider
    \begin{equation*}
        \mathcal F(\omega) \coloneqq\{(x,\theta)\in \whM\times \Tds: \; (x, \theta) = \psi_{\alpha,\omega}(u),\; u\in M\},
    \end{equation*}
    where
    \begin{eqnarray*}
        \psi_{\alpha,\omega}: M &\longrightarrow& \whM\times \Tds\\
        u &\longmapsto& \mleft(\hat\pi(u),\; \omega + F_\alpha(u) \; ({\rm mod.\; \Zds}) \mright).
    \end{eqnarray*}

    \medskip

    Let us define the following family of Hamiltonians:
    \begin{eqnarray*}
        H_\alpha: T^* M \times \Tds &\longrightarrow& \Rds\\
        (u, p, \omega) &\longmapsto& H\mleft(\psi_{\alpha,\omega}(u),p\circ(\drm\psi_{\alpha,\omega}(u))^{-1}\mright).
    \end{eqnarray*}
    Let $\G \simeq \Zds^b$ be the corresponding group of Deck transformations. In particular, $\G = {\rm im}[\pi_1(\whM) \to H_1(M, \Zds)]$, {\it i.e.}, the image of the Hurewicz homomorphism and $b$ equals the rank of $H_1(M,\Zds)$.

    Let us consider $(\Omega, \Sigma, \Pds)$, where $\Omega=\Tds$, $\Sigma$ denotes the $\sigma$-algebra of the Lebesgue measurable subsets of $\Tds$ and $\Pds$ the Lebesgue (probability) measure on $\Tds$. We consider the following maps:
    \begin{eqnarray*}
        \tau: \G \times \Tds &\longrightarrow& \Tds \\
        (h, \omega) &\longmapsto& \tau_h(\omega)\coloneqq\omega + \langle [\alpha], h \rangle \qquad ({\rm mod.}\, \Zds),
    \end{eqnarray*}
    where $[\alpha] \in H^1(\whM,\Rds)$ denotes the cohomology class of $\alpha$ and $\langle \cdot, \cdot \rangle$ the canonical pairing between $H^1(\whM,\Rds)$ and $H_1(\whM,\Rds)\simeq\Rds\otimes \G$. For every $h \in \G$, the map $\tau_h: \Omega \longrightarrow \Omega$ is clearly measurable and preserves $\Pds$.

    We impose the following {\it non-resonant condition on $\alpha$}, namely:
    \begin{equation*}
        \langle [\alpha], h \rangle \not\in \Zds \qquad \forall\; h \in \G \setminus \{e\},
    \end{equation*}
    where $e$ denotes its identity element. This condition is the analogue of the non-resonant condition for the quasi-periodic case.

    Under this condition, it follows (similarly to the quasi-periodic case) that $\{\tau_h\}_{h \in \G}$ is ergodic.

    We can now check that $H_\alpha$ satisfies our assumptions. In fact, for every $h\in \G$ we have (we denote by $h(\cdot)$ the corresponding Deck transformation):

    \begin{equation*}
        F_\alpha(h(u)) - F_\alpha(u) = \int_{\gamma_{h}} \alpha = \langle \alpha, h \rangle
    \end{equation*}
   for every closed curve $\gamma_h$ representing the homology class $h \in H_1(\whM,\Zds)\simeq \G$, thereby
    \begin{align*}
        \drm\psi_{\alpha,\tau_h\omega}(u)&=\Drm(\hat\pi(u),\tau_h\omega+F_\alpha(u))=\Drm(\hat\pi(u),\omega+\langle[\alpha],h\rangle+F_\alpha(u))\\
        &=\Drm(\hat\pi(h(u)),\omega+F_\alpha(h(u)))=\mleft(\drm\hat\pi_{h(u)},\alpha_{h(u)}\mright)\circ\drm h_u=\drm\psi_{\alpha,\omega}(h(u))\circ\drm h_u
    \end{align*}
    and
    \begin{align*}
        H_\alpha(h^*(u,p),\omega)&=H\mleft(\psi_{\alpha,\omega}(h(u)),p\circ(\drm h_u)^{-1}\circ(\drm\psi_{\alpha,\omega}(h(u)))^{-1}\mright)\\
        &=H\mleft(\hat\pi(h(u)),\omega+F_\alpha(h(u)),p\circ(\drm\psi_{\alpha,\omega}(h(u))\circ\drm h_u)^{-1}\mright)\\
        &=H\mleft(\hat\pi(u),\omega+F_\alpha(u)+\langle\alpha,h\rangle,p\circ(\drm\psi_{\alpha,\tau_h\omega}(u))^{-1}\mright)\\
        &=H\mleft(\hat\pi(u),\tau_h\omega+F_\alpha(u),p\circ(\drm\psi_{\alpha,\tau_h\omega}(u))^{-1}\mright)\\
        &=H_\alpha(u,p,\tau_h\omega).
    \end{align*}

    \medskip

    \section{The Stable-like Norm for stationary ergodic Riemannian metrics}\label{sec:stablenorm}

    In this section we would like to describe an application of our previous discussion to the definition of a stable-like norm for families of Riemannian metrics that are defined on non-compact manifolds and are not necessarily periodic. Let us first start by briefly recalling the classical case.

    \subsection{Stable norm of Riemannian metrics on closed manifolds}

    The \textit{stable norm} of a periodic metric was introduced by Federer~\cite{Federer74} and later popularized by Gromov~\cite{GromovLafontainePansu81,BuragoBuragoIvanov01}. It turned out to be a fundamental tool in the study of Riemannian geometry, particularly in the context of homogenization and asymptotic behavior of geodesics, since it provides a way to measure real homology classes in a manner that accounts for the large-scale geometry of the manifold.

    More precisely, let $\widehat M$ be a closed Riemannian manifold with a metric $\g$, and let $H_1(\widehat M, \Rds)$ denote its first real homology group. We define a function on $H_1(\widehat M, \Rds)$ called the \textit{stable norm of $\g$} in the following way:
    \begin{equation*}
        \|h\|_s \coloneqq\inf\mleft\{ \sum |r_i| \, \mathcal L_\g(\gamma_i) \mright\}, \quad \forall h \in H_1(\widehat M, \Rds),
    \end{equation*}
    where $\gamma_i$ are $1$-simplices, $r_i \in \Rds$, $\sum r_i\gamma_i$ is a cycle representing $h$, and $\mathcal L_\g(\gamma_i)$ is the length element induced by the metric $\g$. This function defines a norm on $H_1(\widehat M, \Rds)$.

    The stable norm captures the \textit{asymptotic behavior of minimal lengths} in a given homology class. This reflects a {\it homogenization effect}: at large scales, the Riemannian metric behaves like a norm on homology. We refer to~\cite{GromovLafontainePansu81,BuragoBuragoIvanov01} for a more detailed presentation.

    \medskip

    There is an alternative, equivalent, way to define the stable norm, which is close in spirit to our previous construction. Let us start by considering the integral classes $h \in H_1(\widehat M, \Zds) \simeq \Zds^b \times \Tcal$, where $\Tcal$ denotes the torsion part and $b$ the rank of $H_1(\widehat M, \Zds)$. Then $H_1(\widehat M, \Zds)$ is at finite Hausdorff distance from its $\Zds^b$ component. Such spaces are equivalent for the purposes of the large-scale geometry, which we are concerned with this section. Thus we simply assume that $H_1(\widehat M, \Zds) \simeq \Zds^b$. Then $H_1(\widehat M, \Zds)$ can be represented as a lattice in a $b$-dimensional vector space $V$. For simplicity, one may keep in mind the picture of $\Zds^b$ as the set of integer points in $\Rds^b$, but the Euclidean structure is irrelevant.

    For $h \in H_1(\widehat M, \Zds)$ we define
    \begin{equation*}
        \ell_\g(h) \coloneqq\min \mleft\{ \mathcal L_\g(\gamma) \mid \gamma \text{ is a smooth closed curve representing } h \mright\},
    \end{equation*}

    One can check that for every $h \in H_1(\widehat M, \Zds)$ the following limit exists~\cite[Proposition~8.5.1]{BuragoBuragoIvanov01} (it follows from a simple version of the subadditive ergodic theorem, see~\cite[Lemma~8.5.2]{BuragoBuragoIvanov01} and also~\cite{Liggett85,Levental88}):
    \begin{equation*}
        \|h\|_s\coloneqq\lim_{n\to\infty} \frac{\ell_{\g}(nh)}{n} \qquad n\in\Nds.
    \end{equation*}

    It is easy to check that this function is even, positively homogeneous and satisfies the triangle inequality on $\Zds^b$. We can extend it to $\Qds^b$ defining
    \begin{equation*}
        \|n^{-1} h \|_s \coloneqq\frac{\|h\|_s}{n} \qquad \forall\; n\in \Zds.
    \end{equation*}
    The extended function is also positive homogeneous and satisfies the triangle inequality. Moreover, it is a Lipschitz function on $\Qds^b \subset \Rds^b$ (for example, with respect to the Euclidean norm $|\cdot |$). Indeed, for $h\coloneqq\sum_{i=1}^b \alpha_i e_i \in \Qds^b$ (where $\{e_i\}_{i=1}^b$ are a set of generators of $\Zds^b$) we have:
    \begin{equation*}
        \|h\|_s \le \sum_{i=1}^b |\alpha_i| \|e_i\|_s \le b \,\max_{1\le i \le b}\{\|e_i\|_s\}\,|h|.
    \end{equation*}
    Since $\Qds^b$ is dense in $\Rds^b$, this function has a unique continuous extension to $\Rds^b$, which is indeed a norm.

    \begin{rem}
        In order to prove that the above definitions coincide, it suffices to observe that they coincide on $H_1(\widehat M, \Zds)$. Then, use the density of rational directions.
    \end{rem}

    \begin{rem}
        The stable norm is intrinsically linked to \emph{Mather's minimal average action} (or \emph{Mather's $\beta$-function}) in the context of geodesic flows on Riemannian manifolds. Consider the Lagrangian $L_\g(x,v) = \g_x(v,v)$ associated with the geodesic flow on $(\widehat M, \g)$ and let $\beta_\g: H_1(\widehat M, \Rds) \longrightarrow \Rds$ be the associated $\beta$-function (we refer to~\cite{Sorrentino15} for a more detailed definition of the $\beta$-function). In particular, $\beta_{\g}$ coincides with the effective Lagrangian associated to $L_\g$, corresponding to the one introduced in Section~\ref{effLsec} in the periodic case (see also~\cite{LionsPapanicolaouVaradhan87, ContrerasIturriagaSiconolfi14}).

        The following relation holds~\cite{Massart97,Massart96}:
        \begin{equation*}
            \beta_\g(h) = \|h\|_s^2 \qquad \forall\; h \in H_1(M, \Rds).
        \end{equation*}
    \end{rem}

    \medskip

    \subsection{Stationary ergodic stable norm}

    We would like to define an analogue of the stable norm in the case of Riemannian metrics on non-compact manifolds, that are not necessarily periodic ({\it i.e.}, they are not the lift of a metric defined on a quotient). This is not a trivial task for several reasons. First of all, it is not clear what should replace the role of the (real) first homology group of the manifold, which might be trivial or not catch the full complexity of the manifold at its ``ends''; moreover, dealing with the behaviour of minimizing geodesics when they approach the ``ends'' of the manifold introduce new complexities, and there is no reason why any asymptotic information should be extracted or be available.

    Our setting will satisfy the Standing Assumptions introduced in Section~\ref{sec:2.3}.

    Let $\{\g_\omega\}_{\omega\in \Omega}$ be a family of Riemannian metrics on $M$ satisfying the following properties:
    \begin{enumerate}[label=\textbf{($\g$\arabic*)},ref=\textnormal{\textbf{($\g$\arabic*)}},font=\normalfont]
        \item\label{RMcont} The family $\omega\mapsto \g_\omega(\cdot,\cdot)$ is a continuous function;
        \item\label{RMcoer} There exist two convex nondecreasing superlinearly coercive functions $\vartheta,\Theta:\Rds^+\to\Rds$ such that
        \begin{equation*}
            \vartheta(|v|_x)\le \g_{\omega,x}(v,v) \le\Theta(|v|_x),\qquad \text{for any }(x,v,\omega)\in TM\times\Omega;
        \end{equation*}
        \item\label{RMstat} Stationarity with respect to $\G$:
        \begin{equation*}
            \g_{\tau_g \omega, x}(v,v) = \g_{\omega, g(x)}(\drm g(v), \drm g(v)),\qquad \text{for any $(x,v,\omega)\in TM\times\Omega$ and $g\in \G$}.
        \end{equation*}
    \end{enumerate}

    We consider the Lagrangians $L: TM\times \Omega \longrightarrow \Rds$
    \begin{equation*}
        L(x,v,\omega) \coloneqq\g_{\omega, x} (v,v)
    \end{equation*}
    and the associated Hamiltonians (defined by convex duality) $H: T^*M\times \Omega \longrightarrow \Rds$. It is easy to check that \zcref[comp]{RMcoer,RMcont,RMstat} imply that $H$ satisfies \zcref[comp]{condcoerc,condcont,condconv,condreg,condstat}.

    \begin{rem}\label{defpos}
        We can assume without loss of generality that $\vartheta(|v|_x)>0$ whenever $v\ne0$. Indeed, let $\whTheta$ be the convex conjugate of $\vartheta$ and set
        \begin{align*}
            \ovTheta:\Rds^+&\longrightarrow\Rds,\\
            r&\longmapsto\sup_{\substack{(x,p)\in T^*M,\\\omega\in\Omega,\,|p|_x\le r}}H(x,r,\omega)\le\whTheta(r).
        \end{align*}
        Exploiting the homogeneity of $L$ one can easily prove that $\ovTheta(r)>0$ whenever $r>0$ and similarly that, for the convex conjugate $\ovvartheta$ of $\ovTheta$, $\ovvartheta(s)>0$ whenever $s>0$. We then get that
        \begin{equation*}
            0<\ovvartheta(|v|_x)\le \g_{\omega,x}(v,v),\qquad \text{for any $(x,v,\omega)\in TM\times\Omega$ with $v\ne0$}.
        \end{equation*}
    \end{rem}

    \begin{defin}
        We define the {\it stationary ergodic stable norm} (with respect to $\G$) associated to $\{\g_\omega\}_{\omega\in \Omega}$ the function on $\Rds^b$ given by
        \begin{equation*}
            \|h\|_{s,\G}\coloneqq\overline{L}(h)^{1/2} \qquad \forall\; h\in \Rds^b,
        \end{equation*}
        where $\overline{L}$ is the effective Lagrangian associated to $L(x,v,\omega) \coloneqq\g_{\omega, x} (v,v)$, see Theorem~\ref{effLex}.
    \end{defin}

    \begin{rem}
        One can check that $\|\cdot\|_{s,\G}$ is indeed a norm on $\Rds^b$:
        \begin{itemize}
            \item Since $L\ge 0$, clearly $\|h\|_{s,\G} \ge 0$ for every $h\in \Rds^b$. Moreover, observe that $L=0$ if and only if $v=0$; therefore, using assumption~\ref{RMcoer} and \zcref{defpos} it follows that $\|h\|_{s,\G} = 0$ if and only if $h=0$.
            \item It follows from Proposition~\ref{homogeneityL} and the fact that $L$ is homogeneous of degree $2$ in $v$ that $\|\alpha\, h\|_{s,\G} = |\alpha| \|h\|_{s,\G}$ for every $h\in \Rds^b$ and every $\alpha\in \Rds$.
            \item Clearly, for every $\omega\in\Omega$ and $\varepsilon>0$ the map
            \begin{equation*}
                (x,y)\longmapsto\phi_\varepsilon(x,y,1,\omega)^{1/2}
            \end{equation*}
            defines a Riemannian distance on $M$. Using that
            \begin{align*}
                \phi_\varepsilon((0,k)(x),(\Phi_\varepsilon(h_1+h_2),k)(x),1,\omega)^{1/2} &\le \phi_\varepsilon((0,k)(x),(\Phi_\varepsilon(h_1),k)(x),1,\omega)^{1/2} \\
                &\hphantom{\le}+ \phi_\varepsilon(\Phi_\varepsilon(h_1,k)(x),(\Phi_\varepsilon(h_2+h_1),k)(x),1,\omega)^{1/2},
            \end{align*}
            Proposition~\ref{homogeneityL} and Corollary~\ref{phiunifconv}, we conclude that
            \begin{equation*}
                \|{h_1+h_2}\|_{s,\G}=\ovL\mleft({h_1+h_2}\mright)^{1/2} \le \ovL(h_1)^{1/2} + \ovL(h_2)^{1/2}=\|h_1\|_{s,\G} + \|h_2\|_{s,\G},
            \end{equation*}
            for all $h_1,h_2 \in \Rds^b$.
        \end{itemize}
    \end{rem}

    \section{Auxiliary Results}\label{app:A}

    We store here some technical results which are needed to our analysis.

    \begin{prop}\label{Gtorsion}
        The action group $\G$ is a torsion group if and only if $M$ is compact.
    \end{prop}
    \begin{proof}
        We start assuming that $M$ is compact. The action $\G$ is totally discontinuous and each $g$ is an isometry, hence, fixed an $x\in M$, there is an open ball $B_\delta(x)$ centered at $x$ with radius $\delta>0$ such that
        \begin{equation}\label{eq:Gtorsion1}
            B_\delta(g(x))\cap B_\delta(x)=g(B_\delta(x))\cap B_\delta(x)=\emptyset,\qquad \text{for any $g\in \G$ with $g(x)\ne x$}.
        \end{equation}
        Assuming by contradiction that $\G$ is not a torsion group, we have that there is a $g\in \G$ so that $g^i\ne g^j$ whenever $i\ne j$. It follows from~\eqref{eq:Gtorsion1} that, whenever $i\ne j$,
        \begin{equation*}
            B_\delta\mleft(g^i(x)\mright)\cap B_\delta\mleft(g^j(x)\mright)=g^i(B_\delta(x))\cap g^j(B_\delta(x))=g^j\mleft(g^{i-j}(B_\delta(x))\cap B_\delta(x)\mright)=\emptyset,
        \end{equation*}
        namely
        \begin{equation*}
            \mleft|g^i(x)-g^j(x)\mright|>\delta,\qquad \text{whenever $i\ne j$}.
        \end{equation*}
        This in turn implies that the sequence $\{g^n(x)\}_{n\in\Nds}$ does not admit converging subsequences, in contradiction with the compactness of $M$.\\
        Now assume that $\G$ is a torsion group; since $\G$ is finitely generated, then $\G$ is finite. Consider the identification $M=\G\times M/\G$ and the continuous projection maps
        \begin{equation*}
            \pi_g:M=\G\times M/\G\longrightarrow\{g\}\times M/\G,\qquad \text{for $g\in \G$}.
        \end{equation*}
        It is easy to check that
        \begin{equation*}
            \pi_g\circ\pi^{-1}:M/\G\longrightarrow\{g\}\times M/\G
        \end{equation*}
        is a continuous map for every $g\in \G$, {\it i.e.}, each $\{g\}\times M/\G$ is compact. Since $\G$ is finite, then $\G\times M/\G$ endowed with the topology induced by the union $\bigcup_g\{g\}\times M/\G$ is compact because it is the finite union of compact sets. This topology is finer than the one of $M$, thus $M$ is compact.
    \end{proof}

    \begin{lem}\label{normconebound}
        Given a compact $K\subseteq M$, there exist two positive constants $C_1$ and $C_2$ such that for any $(h,k),(h',k')\in \G$, $x,x'\in K$ and $\varepsilon>0$
        \begin{equation}\label{eq:normconebound.1}
            C_1^{-1}d_\varepsilon\mleft((h,k)(x),\mleft(h',k'\mright)\mleft(x'\mright)\mright)-\varepsilon C_2\le\varepsilon\mleft\|h-h'\mright\|\le C_1d_\varepsilon\mleft((h,k)(x),\mleft(h',k'\mright)\mleft(x'\mright)\mright)+\varepsilon C_2.
        \end{equation}
    \end{lem}
    \begin{proof}
        Fixed $z\in K$ we have, for any $(h,k),(h',k')\in \G$, $x,x'\in K$ and $\varepsilon>0$,
        \begin{multline}\label{eq:normconebound1}
            \mleft|d\mleft((h,k)(x),\mleft(h',k'\mright)\mleft(x'\mright)\mright)-d_z\mleft((h,k),\mleft(h',k'\mright)\mright)\mright|\\
            \begin{aligned}
                &\le d((h,k)(x),(h,k)(z))+d\mleft(\mleft(h',k'\mright)(z),\mleft(h',k'\mright)\mleft(x'\mright)\mright)\\
                &\le d(x,z)+d\mleft(z,x'\mright)\le2\diam(K).
            \end{aligned}
        \end{multline}
        Moreover, we get from \zcref{projepsiso,equivorbit,allnormequiv} that there exist two positive constant $A_1$ and $A_2$ such that
        \begin{equation*}
            A_1^{-1}\mleft\|h-h'\mright\|-A_2\le\ovd_z\mleft(h,h'\mright)-A_2\le d_z\mleft((h,k),\mleft(h',k'\mright)\mright)\le\ovd_z\mleft(h,h'\mright)+A_2\le A_1\mleft\|h-h'\mright\|+A_2.
        \end{equation*}
        The previous inequality and~\zcref{eq:normconebound1} prove~\eqref{eq:normconebound.1}.
    \end{proof}

    \begin{lem}\label{lataux}
        It is given, for every $\delta>0$, an $E_\delta\subseteq\Omega$ with $\Pds(E_\delta)>1-\delta$. There is then a set of full measure $\ovOmega$ such that, fixed an $\omega\in\ovOmega$ and an $\epsilon>0$, there exist two positive constants $R_\epsilon$ and $\delta_\epsilon$ so that, for any $\ovh\in\Rds^b$ with $\epsilon\mleft\|\ovh\mright\|>R_\epsilon$, any open balls centered at $\ovh$ with radius $\epsilon\mleft\|\ovh\mright\|$ must intersect $\mleft\{h\in\Zds^b:\tau_h\omega\in E_{\delta_\epsilon}\mright\}$. Equivalently
        \begin{equation*}
            \mleft\|\ovh-h\mright\|<\epsilon\mleft\|\ovh\mright\|,\qquad \text{for some $\tau_h\omega\in E_{\delta_\epsilon}$}.
        \end{equation*}
    \end{lem}
    \begin{proof}
        Using Zorn's Lemma we write $\{h_i\}_{i\in\Nds}=\Zds^b$ such that $\|h_i\|\le\|h_j\|$ whenever $i<j$. Then we have by Birkhoff's ergodic Theorem (see, {\it e.g.}, \cite[Theorem~20.14]{Klenke14}) that for any $\delta>0$ there is a set of full measure $\Omega_\delta\subseteq\Omega$ such that
        \begin{equation*}
            \lim_{n\to\infty}\frac1n\sum_{i=1}^n\chi_{E_\delta}(\tau_{h_i}\omega)=\Pds(E_\delta),\qquad \text{for any }\omega\in\Omega_\delta,
        \end{equation*}
        where $\chi_{E_\delta}$ is the characteristic function of $E_\delta$. It follows that for each $\delta>0$ and $\omega\in\Omega_\delta$
        \begin{equation}\label{eq:lataux1}
            \frac{\#\{h_i:i\le n,\,\tau_{h_i}\omega\in E_\delta\}}n\ge\Pds(E_\delta)-\delta>1-2\delta,\qquad \text{whenever $n$ is big enough}.
        \end{equation}
        Next, we define the set of full measure $\ovOmega\coloneqq\bigcap\limits_{\delta\in\Qds^+}\Omega_\delta$ and fix $\omega\in\ovOmega$ and a constant $\epsilon>0$.\\
        It is a well known fact that the number of points of $\Zds^b$, with $b\ge1$, lying in a closed ball $\ovB(h,R)\subseteq\Rds^b$ centered in $h\in\Rds^b$ with radius $R$ satisfies the identity
        \begin{equation*}
            \#\mleft(\ovB(h,R)\cap\Zds^b\mright)=R^b\mleft|\ovB(0,1)\mright|+\Orm\mleft(R^a\mright),
        \end{equation*}
        where $\Orm$ is Landau's symbol and $a\le b-1$. It is then easy to see that there is a constant $C\ge1$ so that, whenever $R$ is big enough,
        \begin{equation}\label{eq:lataux2}
            \#\mleft(\ovB(h,R)\cap\Zds^b\mright)<CR^b.
        \end{equation}
        We then choose a rational $\delta_\epsilon\le\dfrac\epsilon{2C^2(2+\epsilon)}$ and $R_\epsilon>0$ such that~\zcref{eq:lataux1,eq:lataux2} hold true for any $n$ with $\epsilon\|h_n\|>R_\epsilon$ and $R>R_\epsilon$, respectively. Let us assume by contradiction that there is a $\ovh\in\Rds^b$ with $\epsilon\mleft\|\ovh\mright\|>R_\epsilon$ so that the closed ball $\ovB\mleft(\ovh,\epsilon\mleft\|\ovh\mright\|\mright)$ does not intersect $\mleft\{h\in\Zds^b:\tau_h\omega\in E_{\delta_\epsilon}\mright\}$. Setting
        \begin{equation*}
            n\coloneqq\mleft\lfloor C(2+\epsilon)^b\mleft\|\ovh\mright\|^b\mright\rfloor\ge C(1+\epsilon)^b\mleft\|\ovh\mright\|^b>\#\mleft(\ovB\mleft(0,(1+\epsilon)\mleft\|\ovh\mright\|\mright)\cap\Zds^b\mright),
        \end{equation*}
        where $\lfloor\cdot\rfloor$ is the floor function, yields that if
        \begin{equation*}
            h_i\in\ovB\mleft(\ovh,\epsilon\mleft\|\ovh\mright\|\mright)\subset\ovB\mleft(0,(1+\epsilon)\mleft\|\ovh\mright\|\mright)
        \end{equation*}
        then $i<n$ and, consequently, $\epsilon\|h_n\|>R_\epsilon$. It follows that, since $\ovB\mleft(\ovh,\epsilon\mleft\|\ovh\mright\|\mright)$ does not intersect $\mleft\{h\in\Zds^b:\tau_h\omega\in E_{\delta_\epsilon}\mright\}$,
        \begin{equation*}
            \frac{\#\{h_i:i\le n,\,\tau_{h_i}\omega\in E_{\delta_\epsilon}\}}n\le\frac{n-\#\mleft(\ovB\mleft(\ovh,\epsilon\mleft\|\ovh\mright\|\mright)\cap\Zds^b\mright)}n\le1-\frac\epsilon{C^2(2+\epsilon)}\le1-2\delta_\epsilon,
        \end{equation*}
        in contradiction with~\eqref{eq:lataux1}.
    \end{proof}

    The next results focus on the minimal action $\phi_\varepsilon$.

    \begin{lem}\label{phibound}
        There exist two convex nondecreasing superlinearly coercive functions $\vartheta_L,\Theta_L:\Rds^+\to\Rds^+$ such that, for any $\varepsilon>0$, $(x',x,t)\in M_\varepsilon^2\times\Rds^+$ and $\omega\in\Omega$,
        \begin{equation}\label{eq:phibound.1}
            t\vartheta_L\mleft(\frac{d_\varepsilon(x',x)}t\mright)\le\phi_\varepsilon\mleft(x',x,t,\omega\mright)\le t\Theta_L\mleft(\frac{d_\varepsilon(x',x)}t\mright).
        \end{equation}
    \end{lem}
    \begin{proof}
        We define $\vartheta_L$ and $\Theta_L$ as the convex conjugate of the maps $\Theta$ and $\vartheta$, respectively, in \zcref{condcoerc}. It is apparent from the definition of $L$ that
        \begin{equation*}
            \vartheta_L(|q|_x)\le L(x,q,\omega)\le\Theta_L(|q|_x),\qquad \text{for any $(x,q)\in TM$ and $\omega\in\Omega$},
        \end{equation*}
        hence is a simple check that, for any $\mleft(x',x,t\mright)\in M^2\in\Rds^+$ and $\omega\in\Omega$,
        \begin{equation}\label{eq:phibound1}
            t\vartheta_L\mleft(\frac{d(x',x)}t\mright)\le\phi\mleft(x',x,t,\omega\mright)\le t\Theta_L\mleft(\frac{d(x',x)}t\mright).
        \end{equation}
        Replacing $t$ with $t/\varepsilon$ in~\eqref{eq:phibound1} proves~\eqref{eq:phibound.1}.
    \end{proof}

    \begin{prop}\label{phiepslip}
        For every positive $A$ and $\varepsilon$ we define the sets
        \begin{equation*}
            \Ccal_{\varepsilon,A}\coloneqq\mleft\{\mleft(x',x,t\mright)\in M_\varepsilon^2\times\Rds^+:d_\varepsilon\mleft(x',x\mright)<At\mright\}.
        \end{equation*}
        There is an $\ell_A>0$, depending on $\vartheta_L$, $\Theta_L$ and $A$, such that $\phi_\varepsilon(\omega)$ is $\ell_A$--Lipschitz continuous in $\overline{\Ccal_{\varepsilon,A}}$ for all $\omega\in\Omega$. In particular, $\{\phi_\varepsilon(\omega)\}_{\varepsilon\in(0,1]}$ is locally $d_\varepsilon$-equiLipschitz continuous in $M_\varepsilon^2\times[t_0,\infty)$ for any $t_0>0$, uniformly with respect to $\omega$.
    \end{prop}
    \begin{proof}
        This result has been proved in~\cite[Theorem~3.1]{Davini07} when $\varepsilon=1$ and~\zcref{eq:minactequiv} extend it to our case. The local $d_\varepsilon$-equiLipschitz continuity follows from that.
    \end{proof}

    \begin{lem}\label{phiepsconvaux}
        We have
        \begin{equation*}
            \lim_{\varepsilon\to0^+}\mleft|\phi_\varepsilon\mleft(\mleft(\Phi_\varepsilon\mleft(h'\mright),k'\mright)\mleft(x'\mright),(\Phi_\varepsilon(h),k)(x),1,\omega\mright)\!-\frac1t\phi_{t\varepsilon}\mleft(\mleft(\Phi_{t\varepsilon}\mleft(th'\mright),k'\mright)\mleft(x'\mright),(\Phi_{t\varepsilon}(th),k)(x),t,\omega\mright)\mright|\!=0
        \end{equation*}
        for any $x,x'\in M$, $k,k'\in\Tcal$, $h,h'\in\Rds^b$, $\omega\in\Omega$ and $t>0$.
    \end{lem}
    \begin{proof}
        To ease the proof we will assume that $h'=0$, the same arguments also prove the general case. Taking into account \zcref{phiepslip,eqcont,eq:minactequiv}, the claim is equivalent to show that
        \begin{equation}\label{eq:phiepsconvaux1}
            \lim_{\varepsilon\to0^+}\varepsilon\mleft|\phi\mleft((0,k)(x),(\Phi_\varepsilon(h),k)(x),\frac1\varepsilon,\omega\mright)-\phi\mleft((0,k)(x),(\Phi_{t\varepsilon}(th),k)(x),\frac1\varepsilon,\omega\mright)\mright|=0.
        \end{equation}
        By \zcref{normconebound,eq:phiepsconv} there exist a $C_1>0$ and an infinitesimal sequence $\{C_\varepsilon\}_{\varepsilon>0}$ of positive numbers so that, for any $\varepsilon>0$,
        \begin{align*}
            d((0,k)(x),(\Phi_\varepsilon(h),k)(x))&\le\frac1\varepsilon(C_1\|h\|+C_\varepsilon),\\
            d((0,k)(x),(\Phi_{t\varepsilon}(th),k)(x))&\le\frac1{t\varepsilon}(C_1\|th\|+C_\varepsilon)=\frac1\varepsilon\mleft(C_1\|h\|+\frac{C_\varepsilon}t\mright).
        \end{align*}
        \zcref[S]{phiepslip,normconebound} then yields that, for some positive constants $\ell$ and $C$,
        \begin{multline*}
            \varepsilon\mleft|\phi\mleft((0,k)(x),(\Phi_\varepsilon(h),k)(x),\frac1\varepsilon,\omega\mright)-\phi\mleft((0,k)(x),(\Phi_{t\varepsilon}(th),k)(x),\frac1\varepsilon,\omega\mright)\mright|\\
            \le\ell d_\varepsilon((\Phi_\varepsilon(h),k)(x),(\Phi_{t\varepsilon}(th),k)(x))\le\ell C\mleft(\|\varepsilon\Phi_\varepsilon(h)-h\|+\frac1t\|t\varepsilon\Phi_{t\varepsilon}(th)-th\|\mright).
        \end{multline*}
        This proves~\eqref{eq:phiepsconvaux1} since the right-hand side of the previous inequality tends to $0$ as $\varepsilon\to0^+$ by~\eqref{eq:phiepsconv}.
    \end{proof}

\end{document}